\documentclass[10pt,a4paper,reqno]{amsart}
\usepackage{amsmath,amsthm,amssymb}
\usepackage{mathrsfs,mathtools,nccmath,enumerate}
\usepackage{tikz-cd}
\usepackage{xcolor}
\usepackage[thinlines]{easytable}
\usepackage{multicol}
\usepackage[mathscr]{euscript}
\usepackage{longtable}
\usepackage{booktabs}
\usepackage{graphicx}
\usepackage{float}
\usepackage{microtype}
\usepackage{subcaption}
\makeatletter
\renewcommand{\p@subfigure}{\thefigure}
\makeatother

\usepackage{hyperref}

\usepackage[super]{nth}

\newtheorem{theorem}{Theorem}[section]
\newtheorem{lemma}[theorem]{Lemma}

\theoremstyle{definition}
\newtheorem{definition}{Definition}[section]
\newtheorem{conjecture}{Conjecture}[section]
\newtheorem{example}{Example}[section]
\theoremstyle{remark}
\newtheorem{remark}[theorem]{Remark}

\usepackage[export]{adjustbox} 

\begin{document}

	\title[A Refinement of the Q-Polynomial for Twisted Knots]{A Refinement of the Q-Polynomial for Twisted Knots}
	
	
	\author[T. MAHATO]{Tumpa Mahato$^*$}
	\address{Department of Mathematics,\\ Indian Institute of Technology, Ropar, Punjab 140001, India}
	\email{staff.tumpa.mahato@iitrpr.ac.in}
	\thanks{*Corresponding author: staff.tumpa.mahato@iitrpr.ac.in}
	
	\author[P. MADETI]{Prabhakar Madeti}
	\address{Department of Mathematics,\\ Indian Institute of Technology, Ropar, Punjab 140001, India}
	\email{prabhakar@iitrpr.ac.in}

	\makeatletter
	\@namedef{subjclassname@2020}{%
		\textup{2020} Mathematics Subject Classification}
	\makeatother
	
%

	\begin{abstract}
This paper introduces a two-variable polynomial invariant for oriented twisted knots, denoted by $Q_{K}^{z}(s,t)$, refining the $Q$-polynomial of N. Kamada and S. Kamada \cite{NaoSei}. We exhibit an infinite family of twisted knots indistinguishable by the $Q$-polynomial but separated by the $Q^z$-polynomial. To prove invariance, we first determine a generating set of oriented Reidemeister moves for twisted knot diagrams, extending the result of Ali \cite{Dan} for oriented virtual knots; this result is new and of independent interest, as it provides the minimal framework needed to verify invariance of any oriented twisted knot invariant. As further applications, we derive an explicit crossing change 
formula, obtain lower bounds on the Gordian distance between homotopic twisted knots, examine the existence of cosmetic crossings in a twisted knot diagram, and finally prove that $Q^{z}_{K}(s,t)$ is a Vassiliev invariant of order one.
	\end{abstract}
	
	\maketitle
	
\section{Introduction}
Twisted links, introduced by Bourgoin \cite{Bur} as a generalization of Kauffman's virtual knot theory  \cite{Kau}, are defined as stable ambient isotopy classes of oriented circles in oriented thickenings that may not be orientable. Twisted link diagrams are defined as marked generic planar curves, where the
markings identify the usual classical crossings, virtual crossings, and bars on edges. Virtual crossings represent places where the link passes over itself along a handle, and bars represent places where an arc of the link moves through a cross-cap. A twisted link with one component is called a twisted knot. New Reidemeister moves $T_{1},T_{2},T_{3}$ are included with the existing generalized Reidemeister moves $C_{1}, C_{2}, C_{3}, V_{1}, V_{2}, V_{3}, V_{4}$ (Figure~\ref{fig:moves}) to define equivalence relation between twisted link diagrams. A trivial twisted link is a link diagram with no classical crossings, having either zero or one bar in each component. We assume throughout that all diagrams are oriented, i.e., each component is equipped with a fixed direction of traversal. 

\begin{figure}[htbp]
	\centering
	\includegraphics[width=0.9\textwidth]{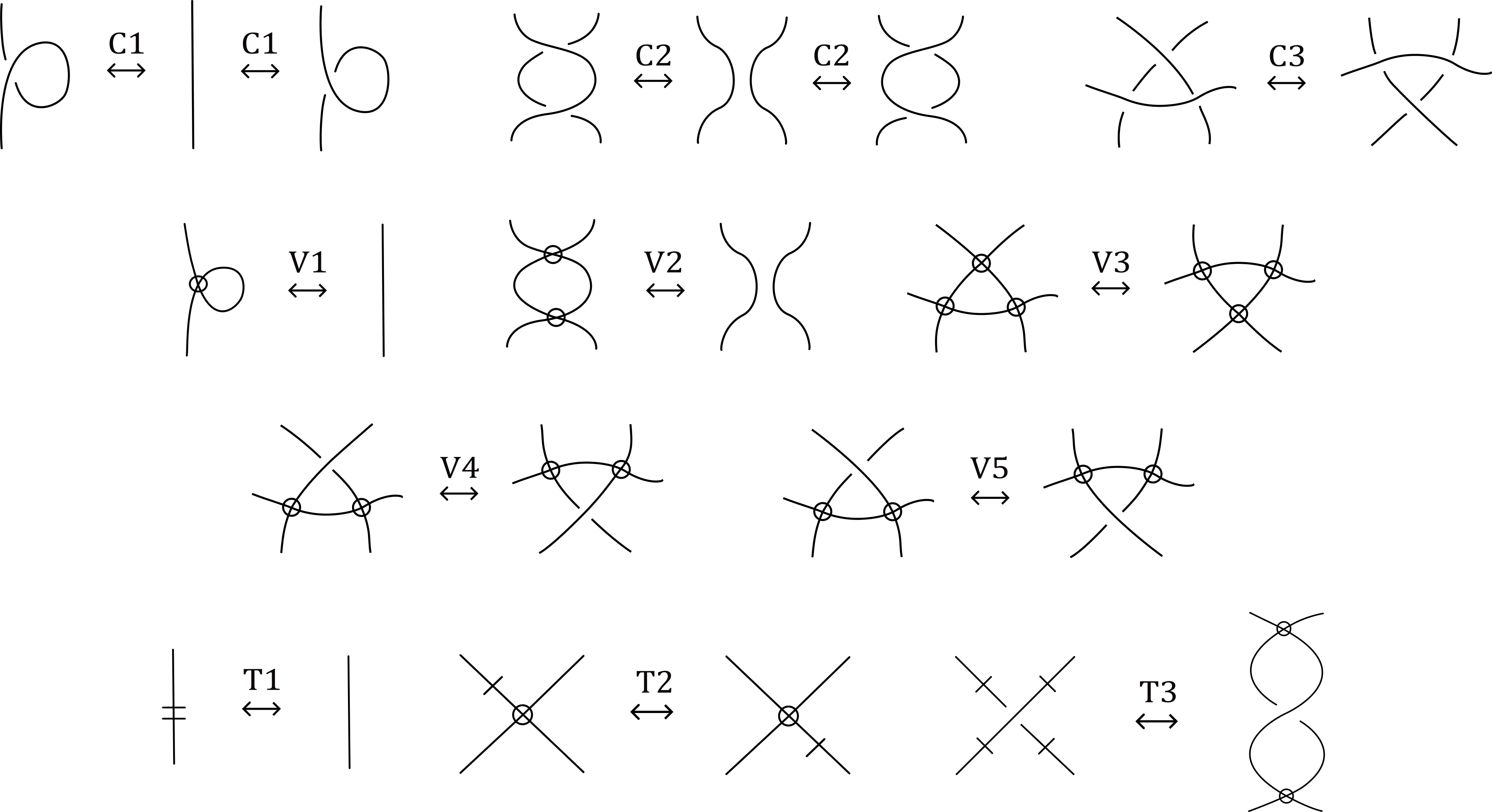}
	\caption{The generalized Reidemeister moves for twisted knot diagrams.}
	\label{fig:moves}
\end{figure} 
The construction of polynomial invariants for virtual and twisted links has been an active area of research since Kauffman's 
foundational work on virtual knot theory \cite{Kau,Kau2}. 
A productive source of such invariants is the \emph{affine index} 
of a crossing. 
The affine index polynomial for virtual knots, defined by 
Kauffman \cite{Kau2} and studied further by Folwaczny--Kauffman 
\cite{Fol}, assigns a Laurent polynomial to each virtual knot 
by summing signed contributions weighted by the affine indices 
of its crossings. Kamada and Kamada extended this construction to 
twisted links and defined the $Q$-polynomial that is invariant under 
all twisted Reidemeister moves \cite{NaoSei}. The $Q$-polynomial is defined as follows: Let $K$ be a twisted link diagram and  denote by $SelfC(K)$, the set of all classical self-crossings. Here a classical self-crossing of $K$ means a classical crossing of $K$ such that the over path and the under path are on the same component of $K$. Let $c$ be a classical self-crossing of
 $K$. Let $K_c$
 be the twisted link diagram obtained from $K$ with components labeled $0$ and $1$, obtained by applying 
 the oriented smoothing at $c$ as shown in Figure~\ref{fig:smoothing}. Then the $Q$-polynomial is defined by
  \[
Q_{K}(s,t) = \sum_{c \in SelfC(K)} sgn(c)\, \bigl(s^{\bar\rho^{0}(c)}t^{p^{0}(c)}-1\bigr)\bigl(s^{\bar\rho^{1}(c)}t^{p^{1}(c)}-1\bigr),
\]
where, for $l \in \{0,1\}$, $\bar\rho^{\,l}(c) \in \{0,1\}$ records the parity of the affine index on the $l$-smoothing component at $c$, and $p^{\,l}(c)\in\{0,1\}$ records the parity of the number of bars on that component.

A limitation of the $Q$-polynomial is that a crossing's contribution to the sum vanishes trivially whenever it has even affine index and an even number of bars on a smoothing component. Thus $Q$-polynomial cannot distinguish two twisted links that differ only by such crossings. This motivates a refinement of the $Q$-polynomial. In this paper, we define a two-variable polynomial $Q^{z}_{K}(s,t)$ for a twisted knot diagram $K$ that is a refined version of $Q$-polynomial for twisted knots. Our construction can be generalized for twisted links too. For this purpose, we call a crossing with even affine index on any component a \emph{zero crossing}, and define a refinement of affine index, called the \emph{zero affine index} for the zero crossings which we use to define $Q^{z}_{K}(s,t)$ for a twisted knot diagram $K$. 

A foundational but so far unresolved problem in twisted knot theory 
is the determination of a \emph{generating set} of oriented 
Reidemeister moves. This problem is non-trivial 
because each unoriented Reidemeister move splits into several oriented versions depending on the relative orientations of the strands involved, and not all of these versions are independent. 
For classical knots, Polyak determined a minimal 
generating set of oriented Reidemeister moves \cite{Polyak}. The analogous problem for virtual knots was 
resolved by Ali \cite{Dan}, who identified a generating 
set of oriented Reidemeister moves for virtual knot diagrams. No analogue of this result existed for twisted knots. 
The absence of such a generating set is an obstacle not just for 
the present paper but for any attempt to define and verify an 
oriented invariant of twisted knots. In this paper, we also provide such a minimal generating set for oriented twisted knot diagrams, showing that every other oriented twisted 
Reidemeister move can be realized as a finite sequence of moves 
from this collection and isotopy of $\mathbb{R}^{2}$.

The paper is organized as follows. In Section~\ref{sec:Qz}, we recall the definition of the affine index of a classical crossing of a twisted knot diagram and define the polynomial $Q^{z}_{K}(s,t)$. In Section~\ref{sec:invariance}, we provide a generating set of oriented Reidemeister moves for twisted knot diagrams extending the result of Ali \cite{Dan} for virtual knots, and prove invariance of $Q^{z}_{K}(s,t)$ under all classical, virtual, and twisted Reidemeister moves. In Section~\ref{sec:examples}, we give examples of twisted knot diagrams with identical
$Q$-polynomials that are distinguished by the $Q^z$-polynomial, and exhibit an infinite family $\{T_n\}_{n \in \mathbb{N}}$ of twisted knots sharing the same $Q$-polynomial but with pairwise distinct $Q^z$-polynomials. Section~\ref{sec:operations} establishes further properties of $Q^{z}_{K}(s,t)$ under orientation reversal and mirror imaging. Section~\ref{sec:applications} uses $Q^{z}_{K}(s,t)$ to give a lower bound on the Gordian distance between homotopic twisted knots, address the cosmetic crossing conjecture for twisted knots, and show that $Q^{z}_{K}(s,t)$ is a Vassiliev invariant of order one.
\section{The \texorpdfstring{$Q^{z}$}{Qz}-polynomial for twisted knots}\label{sec:Qz}
	
First we recall the definition of the affine index of a twisted knot diagram at a crossing, due to 
N.~Kamada and S.~Kamada \cite{NaoSei}.	Let $K$ be an oriented twisted knot
	 diagram. Since $K$ is a knot diagram, every classical crossing of $K$ is necessarily a self-crossing; that is, $C(K) = SelfC(K)$. We therefore use the notation $C(K)$ throughout the remainder of the paper.
	  Let $c \in C(K)$ and $K_{c}$ be the twisted link diagram obtained from $K$, with components labeled $0$ and $1$, by applying 
	 the oriented smoothing at $c$ as shown in Figure~\ref{fig:smoothing}. 
	 \begin{figure}
	 	\centering
	 	\includegraphics[width=0.7\textwidth]{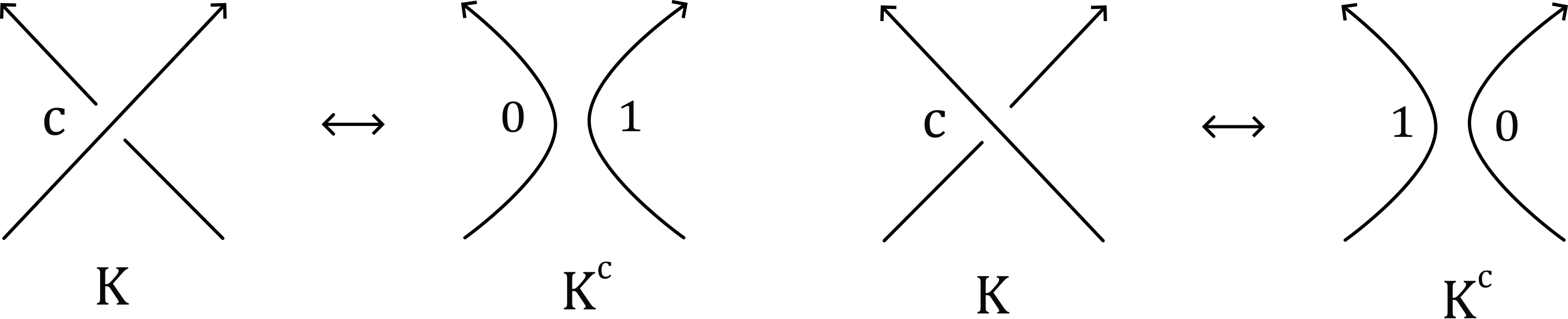}
	 	\caption{Smoothing components at a positive and negative classical self-crossing of $K$.}
 	\label{fig:smoothing}
	 \end{figure}
	 Starting near $c$ 
	on the $l$-component of $K_{c}$ and traversing it in the direction of the 
	orientation, let $\gamma_1, \gamma_2, \ldots, \gamma_{k_l}$ be the 
	sequence of classical crossings encountered in this order before returning to $c$ (Fig.~\ref{fig:ex1}).
	The \emph{affine index} of $K$ at $c$ with respect to the 
	$l$-component is
	\[
	ind^{\,l}(c,K) = \sum_{i=1}^{k_l} (-1)^{\epsilon_{i}}\,sign(\gamma_{i}),
	\]
	where
	\[
	\epsilon_{i} =
	\begin{cases}
		0, & \text{if } \gamma_i 
		\text{ is an undercrossing on the } l\text{-component},\\
		1, & \text{if } \gamma_i 
		\text{ is an overcrossing on the } l\text{-component}.
	\end{cases}
	\]
	Note that self-crossings of the $l$-component appear twice in the 
	sequence $\gamma_1, \ldots, \gamma_{k_l}$ with opposite crossing types, 
	and therefore contribute zero to $ind^{\,l}(c,K),\, l=0,1$.
	
\begin{definition}
	A crossing $c \in C(K)$ is called a \emph{zero crossing} if its 
	affine index with respect to either component of $K_{c}$ is even i.e., $ind^{\,l}(c,K)=0 \pmod2$ for $l=0$ or $1$. 
	We denote the set of all zero crossings of $K$ by $ZC(K)$.
\end{definition}
	
	\subsection{Zero affine index}
	
	For the zero crossings of $K$, we define the {\it zero affine index} with respect to the $l$-component ($l\in \{0,1\}$) by summing contributions only from other zero crossings encountered along the $l$-component. It records how 
	the zero crossings of $K$ contribute to the affine index computation 
	at $c$.
	
	\begin{definition}
		For a crossing $c \in ZC(K)$ and $l \in \{0,1\}$, the 
		\emph{zero affine index} of $K$ at $c$ with respect to the 
		$l$-component is
		\[
		zind^{\,l}(c,K) =  \sum_{\gamma_{i} \in ZC(K)} (-1)^{\epsilon_{i}}\,sign(\gamma_{i}),
		\]
		 where the sum is taken only over those crossings encountered on 
		the $l$-component of $K_{c}$ that are zero crossings of $K$. We then 
		define
		\[
		z^l(c) = zind^{\,l}(c,K)\pmod{2},\quad l \in \{0,1\}.
		\]
	\end{definition}

		\begin{remark}
		Since  $K$ is a twisted knot diagram, $z^{0}(c)=z^{1}(c)$ for $ c \in ZC(K)$. Hence, we only need to compute one of  $z^{0}(c)$ or $z^{1}(c)$. 
	\end{remark}
	For simplicity, let us write $z(c):=z^{0}(c)=z^{1}(c)$.

	\begin{definition}
		Let $K$ be an oriented twisted knot diagram. We define the $Q^z$-polynomial as a refinement of the $Q$-polynomial, denoted by $Q^{z}_{K}(s,t) \in \mathbb{Z}[s,t]$ and is defined by
		\begin{align*}
				Q^{z}_{K}(s,t) 
				=\; & \sum_{c \in ZC(K)} 
				sgn(c)\bigl(s^{z(c)}t^{p^{0}(c)}-1\bigr)
				\bigl(s^{z(c)}t^{p^{1}(c)}-1\bigr) \\
				+\; & \sum_{c \notin ZC(K)} 
				sgn(c)\bigl(s\,t^{p^{0}(c)}-1\bigr)
				\bigl(s\,t^{p^{1}(c)}-1\bigr),
			\end{align*}
			where
			\[
			p^{l}(c) =
			\begin{cases}
				1, & \text{if the number of bars on the $l$-component 
					of $K_{c}$ is odd},\\
				0, & \text{if the number of bars on the $l$-component 
					of $K_{c}$ is even}.
			\end{cases}
			\]
			 for $l \in \{0,1\}$.
	\end{definition}

	\begin{remark}
		The $Q^{z}$-polynomial  coincides with the $Q$-polynomial \cite{Naoko2} in the following cases.
		\begin{enumerate}
			\item If $ZC(K) = \emptyset$, i.e., every crossing of $K$ has odd affine index, then every crossing $c \in C(K)$ has odd affine index on both smoothing components. So the exponent of $s$ in the definition of $Q_{K}(s,t)$ equals $1$ for every crossing. This is exactly the exponent used in the definition of $Q^{z}_{K}(s,t)$ for crossings not in $ZC(K)$.Since $ZC(K)=\emptyset$, every crossing of $K$ falls into this case, and we get
				\begin{align*}
				Q^{z}_{K}(s,t) 
				=\,& \sum_{c \in C(K)} 
				sgn(c)\bigl(s\,t^{p^{0}(c)}-1\bigr)
				\bigl(s\,\,t^{p^{1}(c)}-1\bigr)
				=\, Q_{K}(s,t).
			\end{align*}
				\item If $|ZC(K)| = 1$, say $ZC(K) = \{c_0\}$, then there is no other zero crossings in any component of $K_{c_0}$ and therefore
			$ind^{\,l}(c_{0},K)= zind^{\,l}(c_{0},K)=0$ for $ l\in \{0,1\}$. Then, the term corresponding to $c_0$ in both $Q_{K}(s,t)$ and $Q_{K}^{z}(s,t)$ is	$sgn(c_{0})\bigl(t^{p^{0}(c_{0})}-1\bigr)
			\bigl(\,t^{p^{1}(c_{0})}-1\bigr)$.
			\item If $ZC(K) = C(K)$, i.e., every classical crossing of $K$ is a 
			zero crossing. Then \[zind^{\,l}(c,K)= ind^{\,l}(c,K),\, \text{for}\, l\in \{0,1\}.\] So, the contribution from each crossing in $C(K)$ is same in both $Q_{K}(s,t)$ and $Q^{z}_{K}(s,t)$.
		
		\end{enumerate}
	\end{remark}
\begin{example}\label{ex:ex1}
	Let $K$ be a twisted knot diagram with $C(K)=\{c_{1}, c_{2}, c_{3}, c_{4}\}$ and the diagrams $K_{c_{i}},\,i=1,2,3,4$ as shown in Fig.~\ref{fig:ex1},\ref{fig:ex1sm}.
	Here, $$ind^{\,0}(c_1,K) = -sgn(\gamma_{1})-sgn(\gamma_{2}) +sgn(\gamma_{3})+sgn(\gamma_{4})= 0.$$
	 Similarly, we compute all values of $ind^{\,0}(c_i,K),i=1,2,3,4$ (Table~\ref{tab:ex1tab}). Therefore, the $Q$-polynomial of $K$ is
	 \[Q_{K}(s,t)= (st-1)^{2}+ (s-1)^{2}.\]
	Now, $ZC(K)=\{c_1, c_4\}$ and 
	\begin{align*}
	zind^{\,0}(c_1,K)=& sgn(c_4) =+1,\\
	zind^{\,0}(c_4,K)=& -sgn(c_1)=-1.
	\end{align*}

	 Hence,  \[Q^{z}_{K}(s,t)= (st-1)^{2}+ 3(s-1)^{2}.\]

	\begin{center}
		\begin{minipage}[c]{0.4\textwidth}
			\centering
			\includegraphics[width=0.47\textwidth]{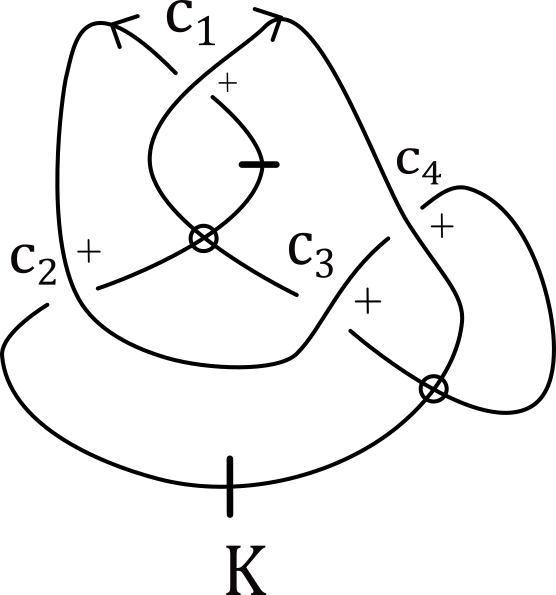}
			\captionof{figure}{}
			\label{fig:ex1}
		\end{minipage}
		\quad
		\begin{minipage}[c]{0.55\textwidth}
			\centering
			\small
			\begin{tabular}{|c|c|c|c|c|}
				\hline
				$c$ & $sgn(c)$ & $ind^{\,0}(c,K)$ & $p^{0}(c)$ & $p^{1}(c)$ \\
				\hline
				$c_{1}$ & $+1$ & $0$  & $0$ & $0$ \\
				\hline
				$c_{2}$ & $+1$ & $1$  & $1$ & $1$ \\
				\hline
				$c_{3}$ & $+1$ & $-1$ & $0$ & $0$ \\
				\hline
				$c_{4}$ & $+1$ & $0$  & $0$ & $0$ \\
				\hline
			\end{tabular}
			\captionof{table}{}
			\label{tab:ex1tab}
		\end{minipage}
	
	\end{center}
	\begin{figure}[h]
			\centering
			\includegraphics[width=0.9\textwidth]{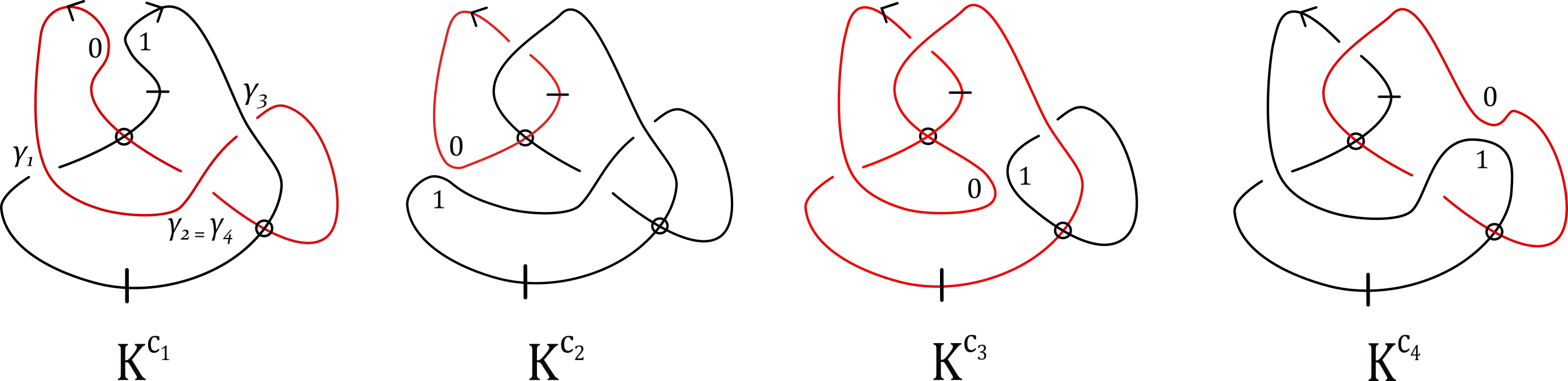}
			\caption{}
			\label{fig:ex1sm}
		\end{figure}	
\end{example}		
	
	\section{Invariance of \texorpdfstring{$Q^{z}_{K}(s,t)$}{Qz\_K(s,t)}}\label{sec:invariance}
To show that $Q^{z}_{K}(s,t)$ is a twisted knot invariant, we show that it is invariant under all oriented classical, virtual and twisted Reidemeister moves. Depending on the orientation of the strands involved in the move, there are several versions of each move. Therefore, first we will determine a generating set for the oriented Reidemeister moves. The following theorem already provides a generating set of Reidemeister moves for the oriented virtual knots. 
\begin{theorem}[\cite{Dan}]
	Let $V$ and $V'$ be two oriented virtual knot diagrams in $\mathbb{R}^2$. Then the diagram $V$ can be transformed into $V'$ by isotopy, and a finite sequence of oriented Reidemeister moves $C1a,C1b, C2a, C3a, V1a, V2a, V3a$ and $V4g$ (See Fig.~\ref{fig:Cl_gen},\ref{fig:vir_genset}).
\end{theorem}

\begin{figure}[htbp]
	\centering
	\includegraphics[width=0.95\textwidth]{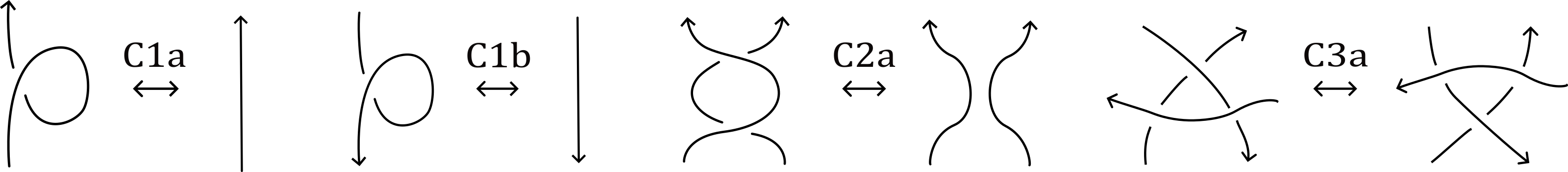}
	\caption{Generating set for classical Reidemeister moves.}
	\label{fig:Cl_gen}
\end{figure}
\begin{figure}[htbp]
	\centering
	\includegraphics[width=\textwidth, height=0.12\textwidth]{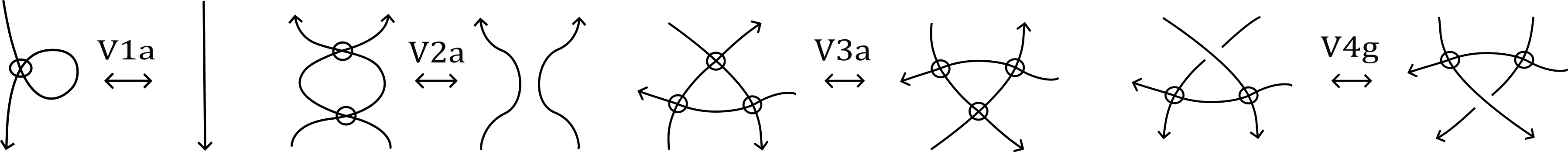}
	\caption{Virtual Reidemeister moves in the generating set.}
	\label{fig:vir_genset}
\end{figure}
\begin{figure}[htbp]

	\centering
	\includegraphics[width=0.9\textwidth]{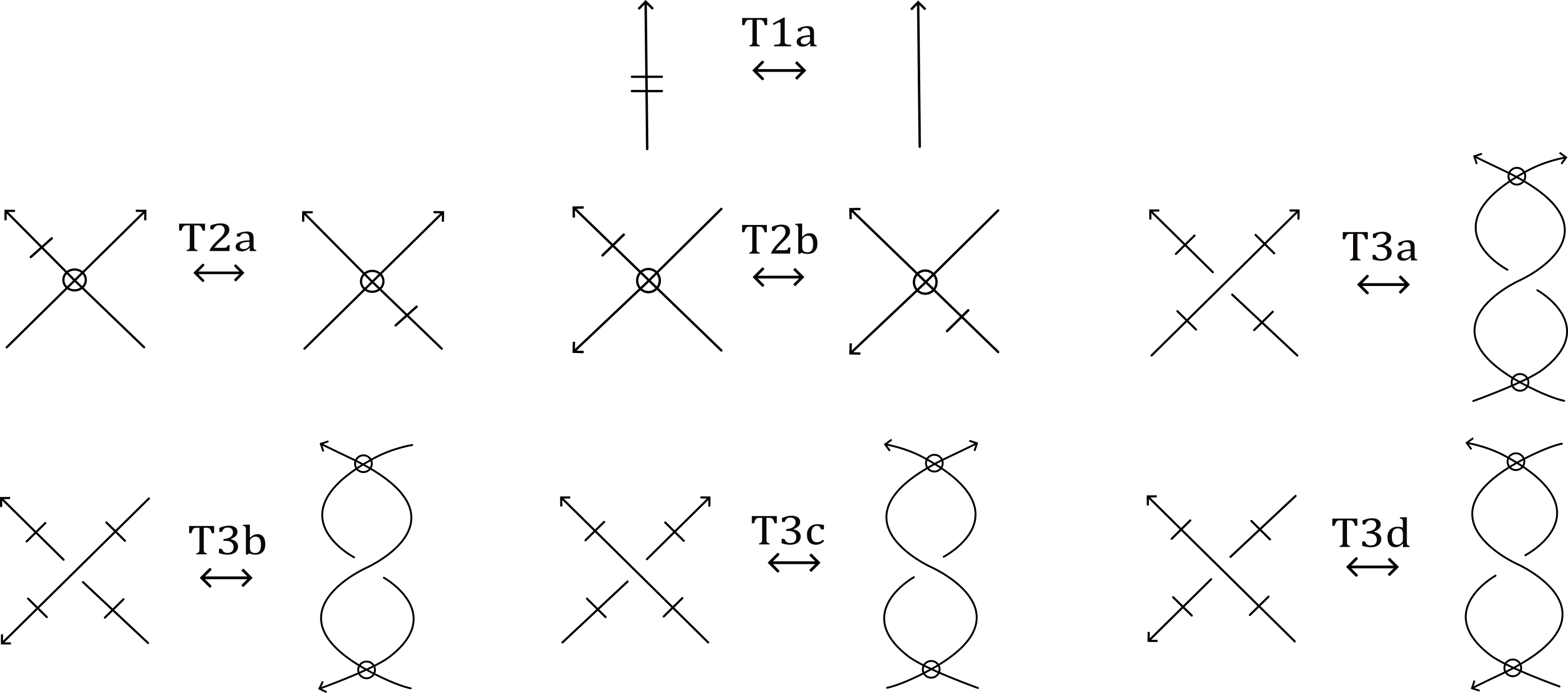}
	\caption{All oriented twisted moves for twisted knot diagrams.}
	\label{fig:tw_genset}.
\end{figure}
\subsection{A generating set for Reidemeister moves of oriented twisted knot diagrams:}
Let $S$ be the set of all generalized Reidemeister moves for oriented twisted knot diagrams. We claim that $$G=\{C1a,C1b, C2a, C3a, V1a, V2a, V3a, V4g, T1a, T2a, T3a, T3c\}$$ is a generating set for $S$. To prove this it is enough to show that the twisted oriented moves in $S$ (Fig.~\ref{fig:tw_genset}) can be realized by a finite sequence of the moves in $G$.

\begin{lemma}\label{lem:t2b}
	The $T2b$ move can be realized by a sequence of the moves $V1a$, $V2a$, $V3a$, and $T2a$.
\end{lemma}
\begin{proof}
In Fig.~\ref{fig:genT2b}, we see that $T2b$ can be obtained by the sequence of the moves $V1b, V2a, V2b, V3b$ and $T2a$. By Lemmas~$2.1, 2.2$, and $2.5$ in \cite{Dan}, $V1b, V2b$,and $V3b$ can be realized by a sequence of the moves $V1a, V2a, V3a$. Hence, the proof.
\end{proof}

\begin{figure}[htbp]
	\centering
	\includegraphics[width=0.85\textwidth]{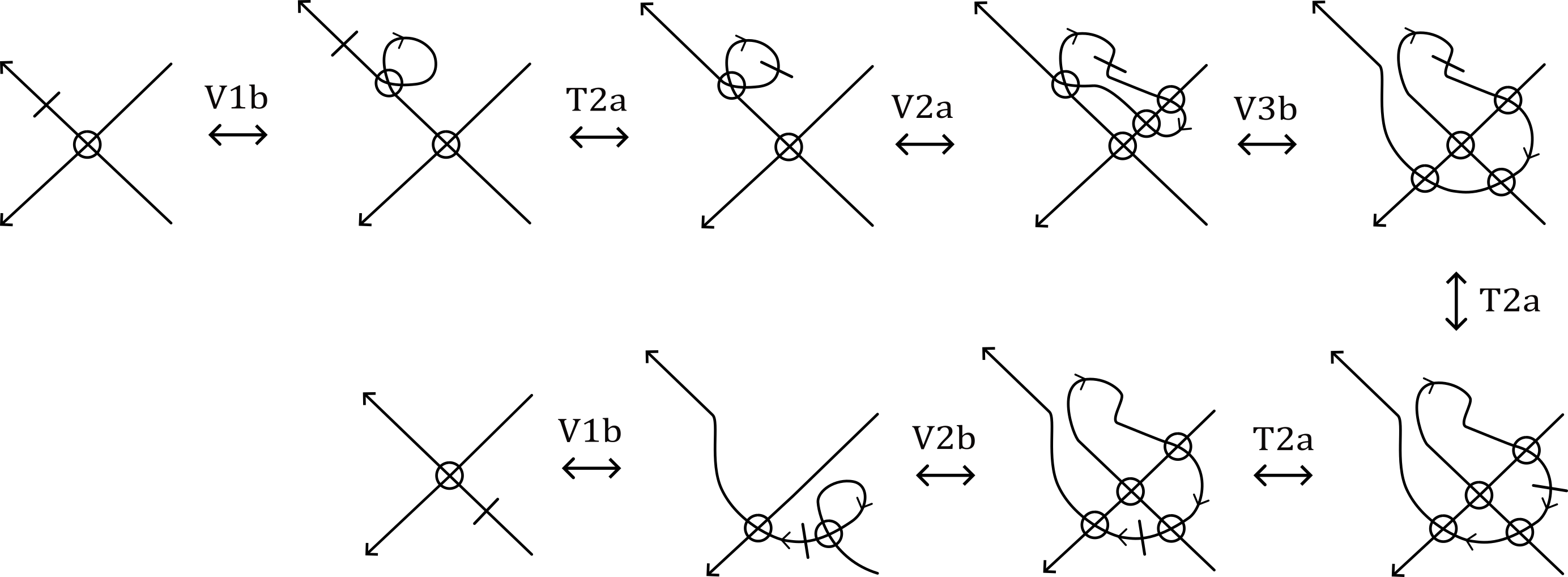}
	\caption{The $T2b$ move is realized by a sequence of $V1b$, $V2a$, $V2b$, $V3b$, $T2a$ moves.}
	\label{fig:genT2b}
\end{figure}
\begin{lemma}\label{lem:t3b}
	The $T3b$ move can be realized by a sequence of the moves $V1a$, $V2a$, $V3a$, $V4g$ and $T3c$.
\end{lemma}
\begin{proof}
In Fig.~\ref{fig:genT3b}, we see that $T3b$ can be obtained by the sequence of moves $V1b, V2a, V2b$, $V3b, V4a, V4g$ and $T3c$. By Lemmas~$2.1, 2.2,2.3, 2.5$, and $2.6$ in \cite{Dan}, $V1b, V2b$, $V3b$, and $V4a$  can be realized by a sequence of the moves $V1a, V2a, V3a$. Hence, the proof.
\end{proof}
	\begin{figure}[htbp]
	\centering
	\includegraphics[width=0.9\textwidth]{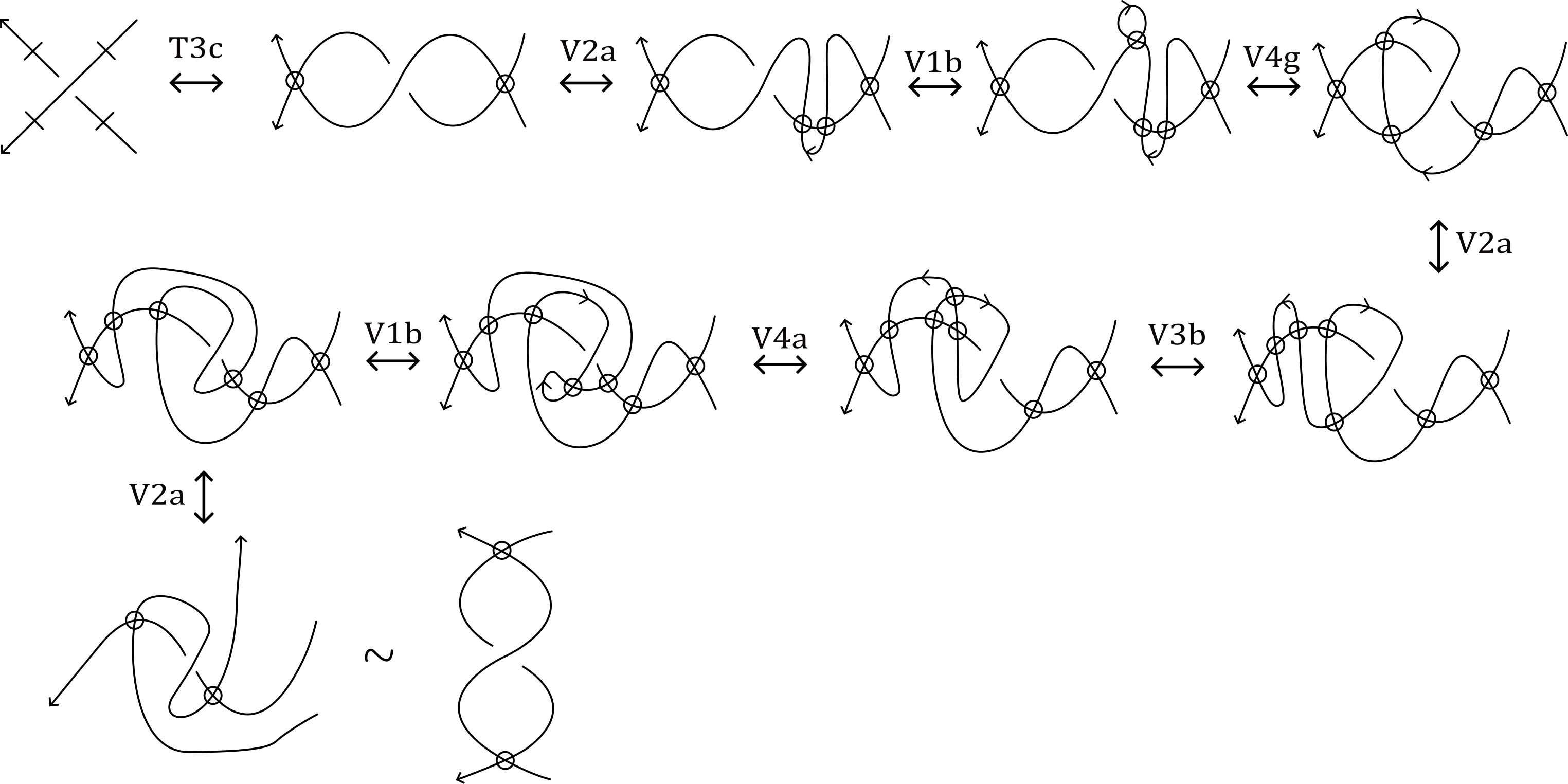}
	\caption{The $T3b$ move is realized by a sequence of $V1b$, $V2a$, $V2b$, $V3b$, $V4a$, $V4g$ and  $T3c$ moves.}
	\label{fig:genT3b}
\end{figure}
\begin{lemma}\label{lem:t3d}
	The $T3d$ move can be realized by a sequence of the moves $C1a$, $C1b$, $C2a$, $V1b$, $V2a$, $V2b$, $V3b$, and  $T3a$.
\end{lemma}
\begin{proof}
	In Fig.~\ref{fig:genT3d}, we see that $T3d$ can be obtained by the sequence of the moves $V1b$, $V2a$, $V3b$, $V4c$, $V4e$ and $T3a$. By results in \cite{Dan} and \cite{Polyak}, we know that $V1b, V3b$, $V4c$, and $V4e$  can be realized by a sequence of the moves $C1a$, $C1b$, $C2a$, $V1a$, $V2a$, $V3a$, and $V4g$. Hence, the proof.
	\begin{figure}[htbp]
		\centering
		\includegraphics[width=0.9\textwidth]{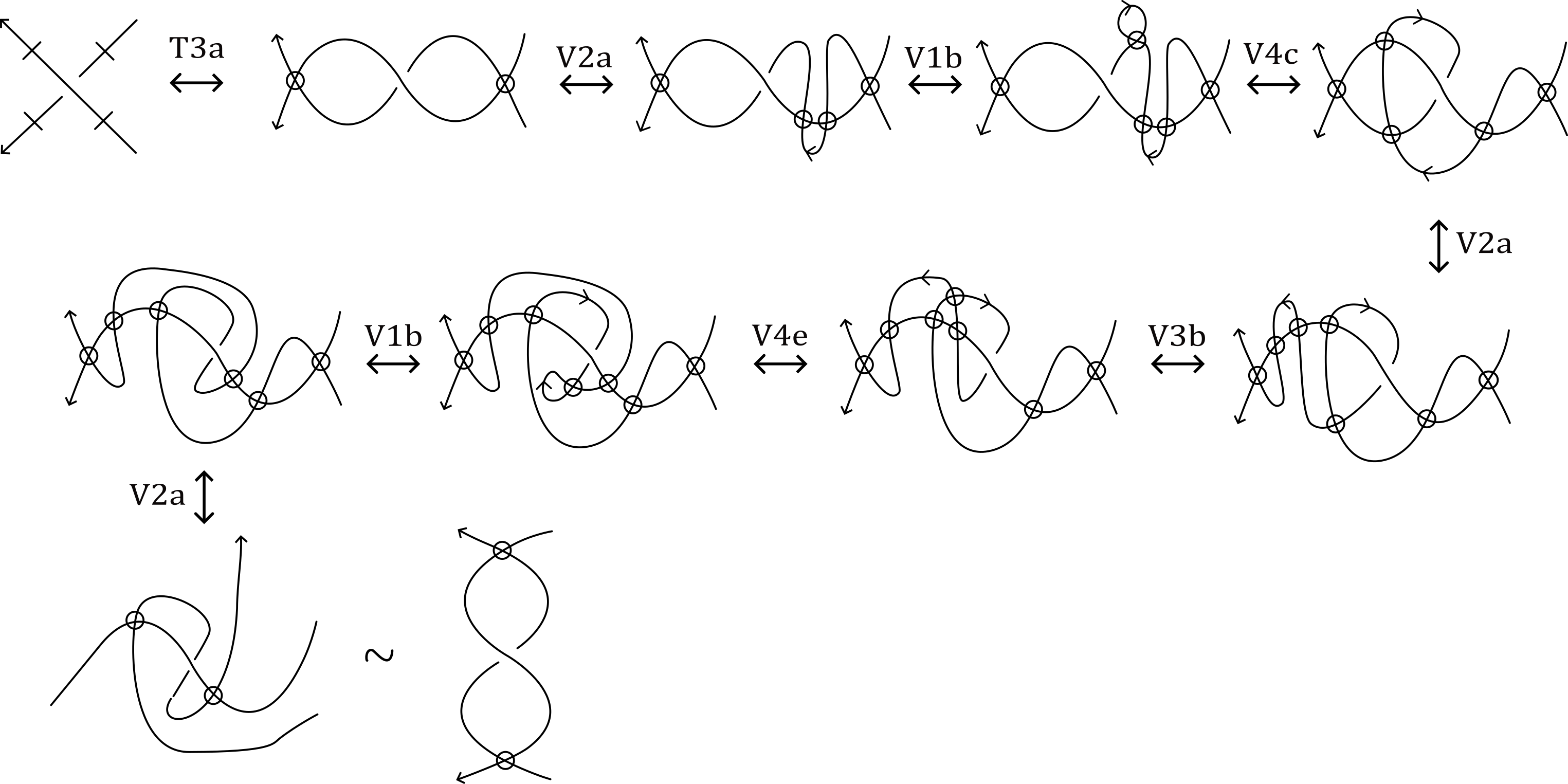}
			\caption{The $T3d$ move is realized by a sequence of $V1b$, $V2a$, $V3b$, $V4c$, $V4e$, and $T3a$ moves.}
		\label{fig:genT3d}
	\end{figure}
\end{proof}
By Lemmas~\ref{lem:t2b}, \ref{lem:t3b}, \ref{lem:t3d}, we can conclude the following.
\begin{theorem}\label{th:gentw}
	Let $K$ and $K'$ be two oriented twisted knot diagrams in $\mathbb{R}^2$. Then the diagram $K$ can be transformed into $K'$ by isotopy, and a finite sequence of oriented Reidemeister moves $C1a$, $C1b$, $C2a$, $C3a$, $V1a$, $V2a$, $V3a$, $V4g$ and $T1a$, $T2a$, $T3a$, $T3c$.
\end{theorem}
Now, we prove the invariance of the $Q^z$-polynomial for the moves in the generating set.

\begin{figure}[htbp]
	\centering
	\includegraphics[width=0.8\textwidth]{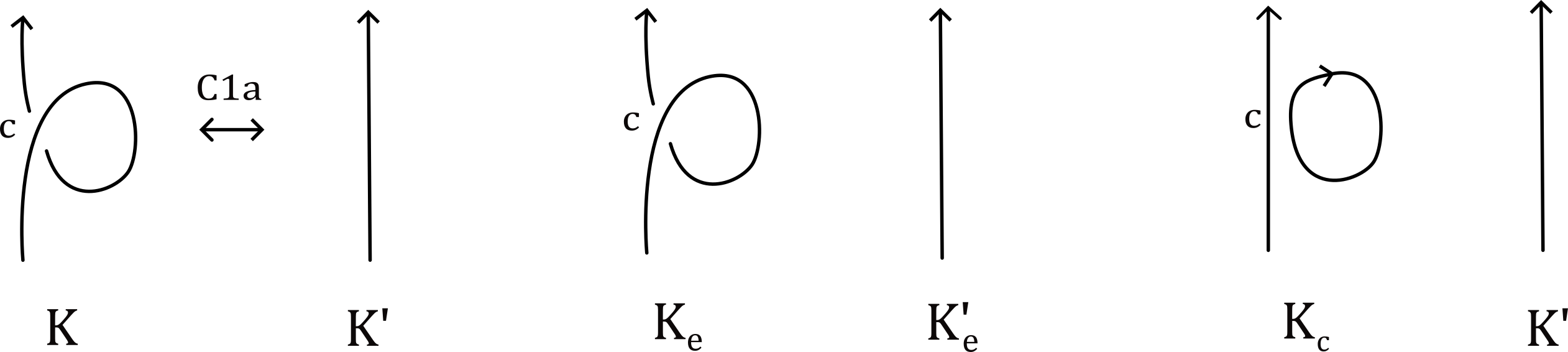}
	\caption{The twisted knot diagrams $K$ and $K'$ related by the oriented Reidemeister move $C1a$ and their smoothing components.}
	\label{fig:C1a}
\end{figure}
\begin{lemma}\label{lem:c1inv}
	$Q_{K}^{z}(s,t)$ is invariant under the $C1a$ move and the
	 $C1b$ move.
\end{lemma}
\begin{proof}
	Suppose $K$ and $K'$ are related by $C1a$ move as in Fig.~\ref{fig:C1a}, where we assume $C(K)= C(K')\cup \{c\}$. Note that $z(c)=0$  since every crossing is a self-crossing in the $0$-component of $K_{c}$ and
	$p^{l}(c) =0$ for at least one $l \in \{0,1\}$. Now, for all crossings $e \in
	C(K)\setminus \{c\}= C(K')$, we see that
	\[ind ^{\,l}(e,K)= ind^{\,l}(e,K'),\; l \in \{0,1\}.\]
	Therefore, $e \in ZC(K)$ if and only if $e \in ZC(K')$.
	Moreover, Since $c$ is a self-crossing, for all $e(\neq c) \in ZC(K)$, $z(e)$ do not change..
	Also, for $l \in \{0,1\}$, $p^{\,l}(e)$ do not change. Therefore,
	\[Q^{z}_{K}(s,t) - Q^{z}_{K'}(s,t)= sgn(c)(s^{z(c)}t^{p^{0}(c)}-1)(s^{z(c)}t^{p^{1}(c)}-1)=0.\]
	Hence, $Q^{z}_{K}(s,t)$ is invariant under  $C1a$ move. Similarly, we can show that it is invariant under  $C1b$ move.		
\end{proof}

\begin{lemma}\label{lem:c2inv}
	$Q^{z}_{K}(s,t)$ is invariant under the $C2a$ move.
\end{lemma}
\begin{proof}
	Suppose $K$ and $K'$ are related by the $C2a$ move as in Fig.~\ref{fig:C2a}, where $C(K)= C(K') \cup \{c_{1},c_{2}\}$ and $sgn(c_{1})= -sgn(c_{2})$.
	
	\begin{figure}[htbp]
		\centering
		\includegraphics[width=0.3\textwidth]{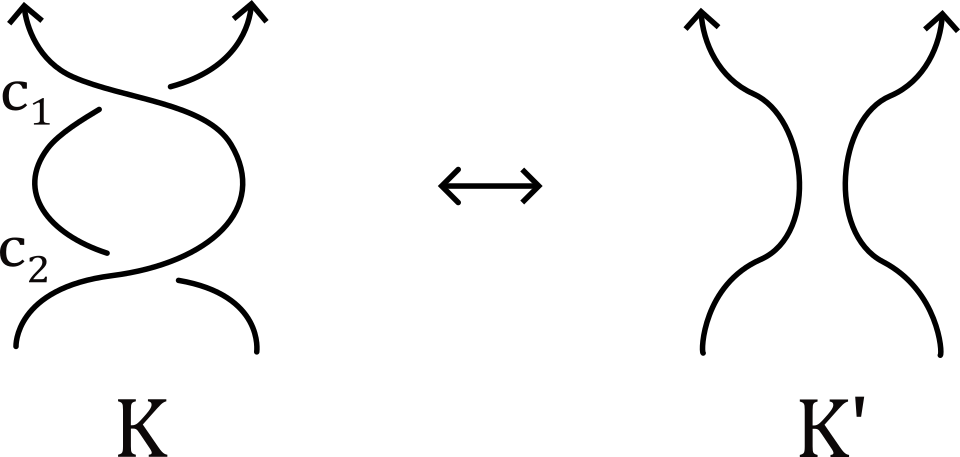}
		\caption{The twisted knot diagrams $K$ and $K'$ related by the oriented Reidemeister move $C2a$.}
		\label{fig:C2a}
	\end{figure}

	For the crossings $c_{1}, c_{2} \in C(K)$, $K_{c_{1}}$ and $K'_{c_{2}}$ are similar except locally as shown in Fig.~\ref{fig:sm1}.  Now, for $l \in \{0,1\}$,
	\begin{equation}
		|ind^{\,l}(c_{1},K)- ind^{\,l}(c_{2},K')| =|-sgn(c_{2})+sgn(c_{1})|=2.
	\end{equation}
	Therefore, $c_{1} \in ZC(K)$ if and only if $c_{2} \in  ZC(K)$. 
		\begin{figure}[htbp]
		\centering
		\subfloat[\label{fig:sm1}]{\includegraphics[width=0.28\textwidth]{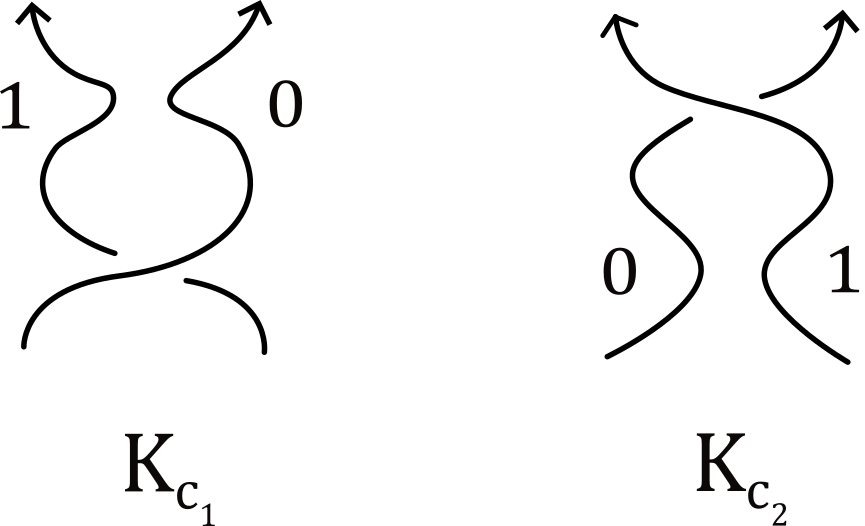}}
		\qquad
		\subfloat[\label{fig:sm2}]{\includegraphics[width=0.33\textwidth]{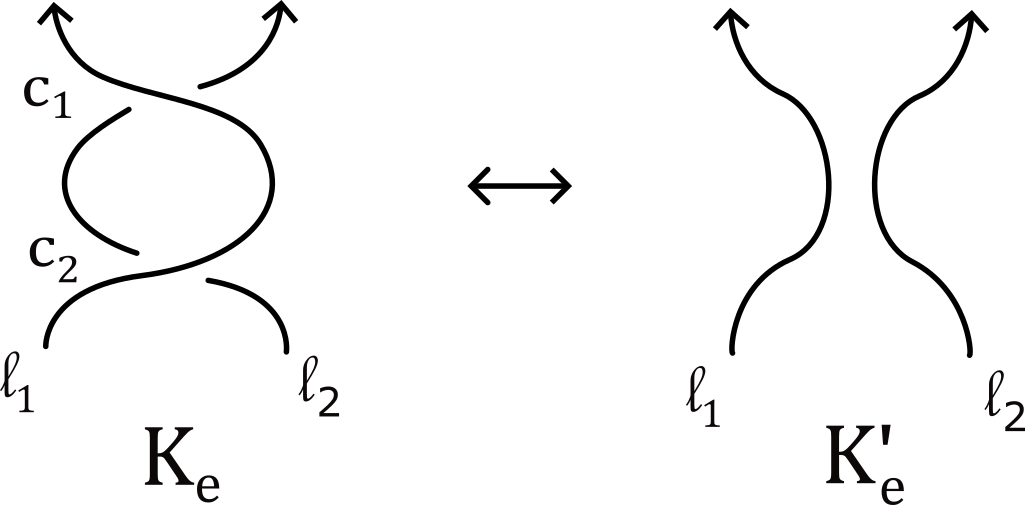}}
		\caption{Smoothing components of the diagram which are related by $C2a$ move.}
		\label{fig:sm}
	\end{figure}

	For all the crossings $e \in C(K) \setminus \{c_{1}, c_{2}\} = C(K')$,  there are two cases. First case is when the strands involved in the move belong to the same component,  and other case when they belong to different components (Fig.~\ref{fig:sm2}). Then either $c_{1}, c_{2}$ are self-crossings or the signs of $c_{1}$ and $c_{2}$ cancels each other. Therefore, in both cases, for every crossing $e \in C(K)\setminus \{c_{1}, c_{2}\}=C(K')$,
	\begin{align*}
		ind ^{\,l}(e,K) &= ind ^{\,l}(e,K') 
		\quad \text{for}\quad l \in \{0,1\},
	\end{align*}
	which implies $e \in ZC(K)$ if and only if $e \in  ZC(K')$.
	Therefore, $ z^{l}(e)$ and $p^{\,l}(e)$, $(l \in \{0,1\})$ do not change when $e \in ZC(K)$. 
	
	Note that 
	\begin{align*}
	|zind^{\,l}(c_{1},K)- zind^{\,l}(c_{2},K)|&= |sgn(c_{2})+sgn(c_{1})|= 0,\, \text{for}\; l \in \{0,1\},
	\end{align*}
	which implies $z^{l}(c_{1})=z^{l}(c_{2}), (l \in \{0,1\})$. Furthermore, $p^{\,l}(c_{1}) = p^{\,l}(c_{2}), \, l \in \{0,1\}$.
	Since $sgn(c_1)=-sgn(c_2)$, we see that the contribution of $c_1$ cancels with that of $c_2$ and we have $Q^{z}_{K}(s,t)=
	Q^{z}_{K'}(s,t)$.
	Hence, $Q^{z}_{K}(s,t)$ is invariant under  $C2a$ move. 
\end{proof}

The following lemmas will be used for the case of Reidemeister move $C3a$. Let denote the crossings involved in $C3a$ as shown in Fig.~\ref{fig:3crossings}.
\begin{figure}[htbp]
	\centering
	\includegraphics[width=0.45\textwidth]{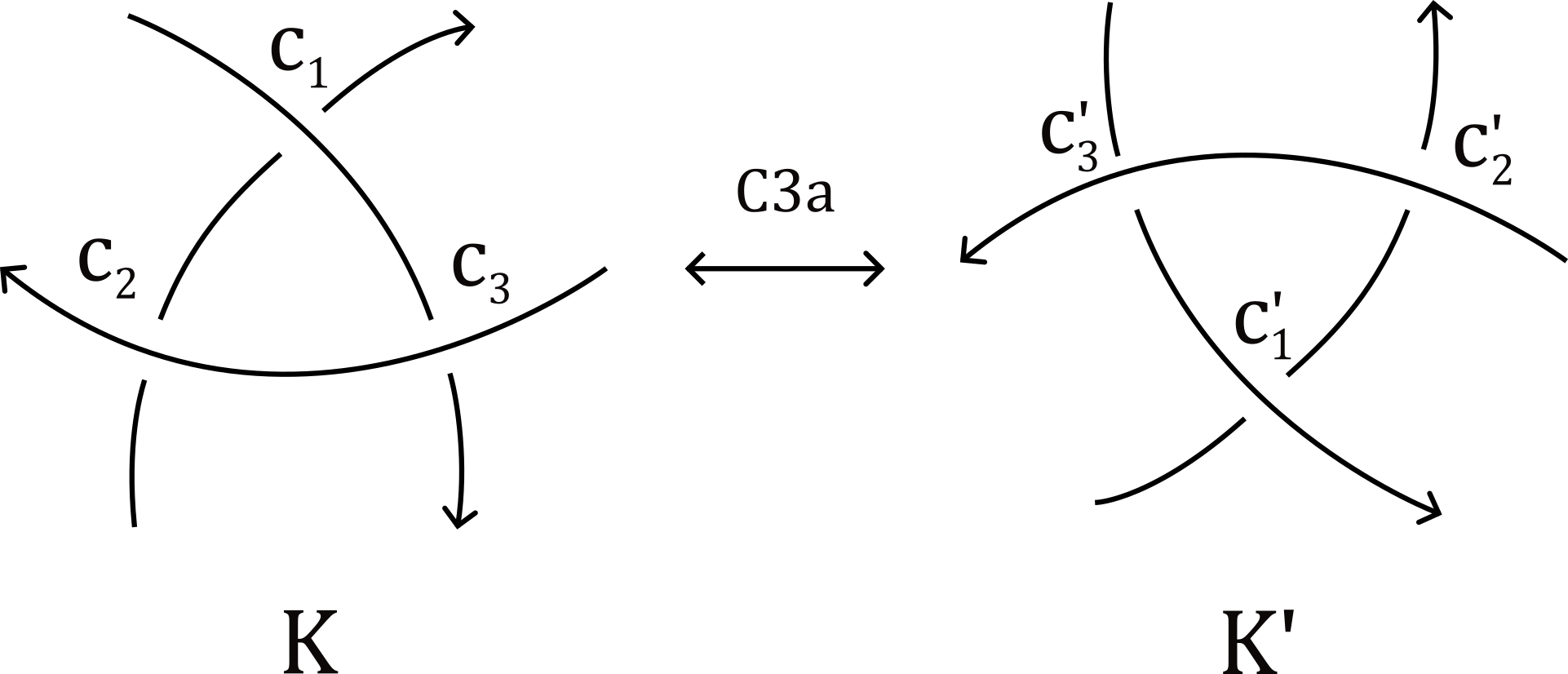}
	\caption{The twisted knot diagrams $K$ and $K'$ related by the Reidemeister move $C3a$.}
	\label{fig:3crossings}
\end{figure}

Now, we will show that if two crossings involved are
zero crossings then the other crossing is also a zero crossing on each side of the $C3a$ move .
	\begin{figure}[htbp]
	\centering
	\includegraphics[width=0.55\textwidth]{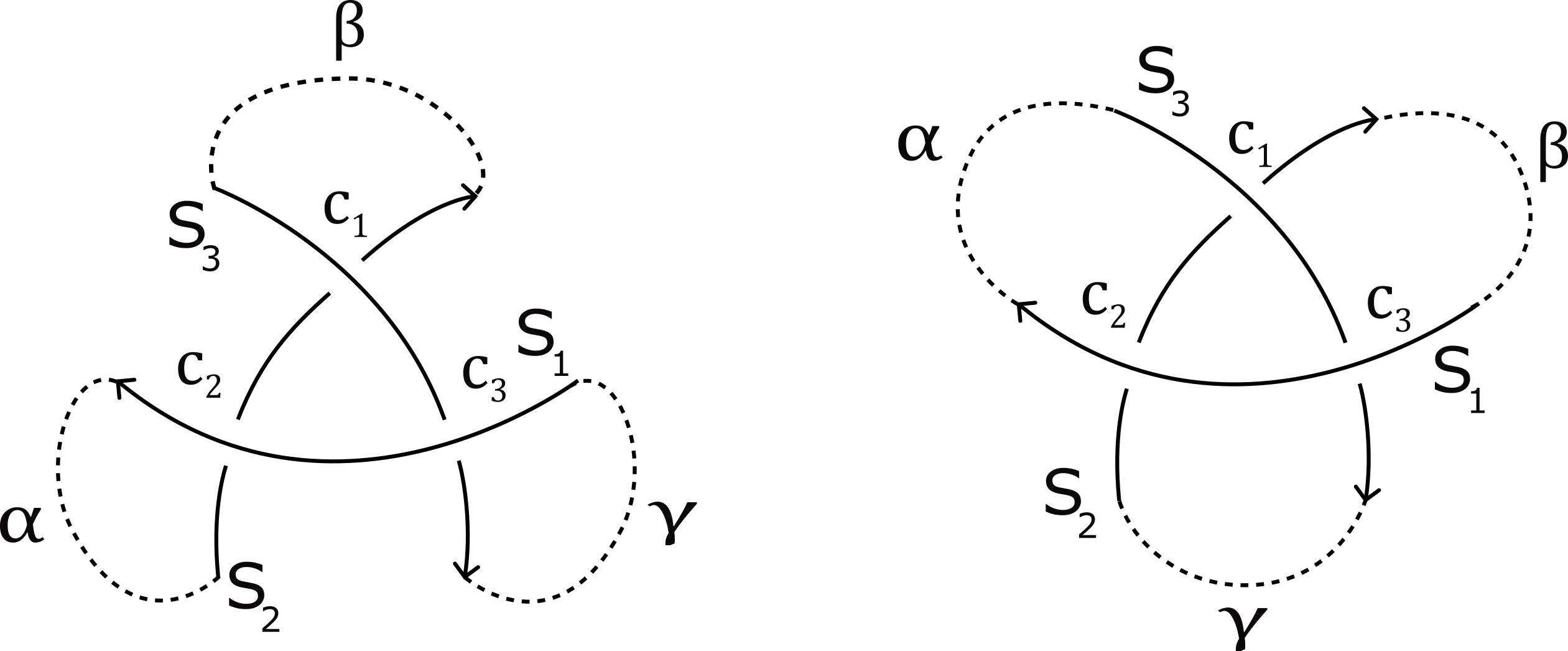}
	\caption{Sequences $(S_{1}S_{2}S_{3})$ and $(S_{1}S_{3} S_{2})$ that gives a twisted knot diagram.}
	\label{fig:3strands}
\end{figure}
\begin{lemma}\label{lem:r31}
	
	Let $\{i,j,k\}=\{1,2,3\}$. Consider the local diagram on the left-hand side of the $C3a$ move (Fig.~\ref{fig:3crossings}) for a twisted knot diagram $K$. 
	If $c_i,c_j \in ZC(K)$, then $c_k \in ZC(K)$.
\end{lemma}

\begin{proof}
	There are two possible ways to reconnect the three strands $S_{1}, S_{2}, S_{3}$ involved in the $C3a$ move to get a global diagram of $K$ (Fig.~\ref{fig:3strands}). The possible sequences are $ (S_{1}S_{2}S_{3})$ or $(S_{1}S_{3} S_{2})$. Let $\alpha,\beta,\gamma$ denote the three strands
	outside the local region.

	The diagrams after smoothing $K$ at  $c_{1}, c_{2}, c_{3}$ with $0$ and $1$-components are as shown in Fig.~\ref{fig:Kc123},~\ref{fig:Kc132}.
	\begin{figure}[htbp]
		\centering 
		\subfloat{\includegraphics[width=0.25\textwidth]{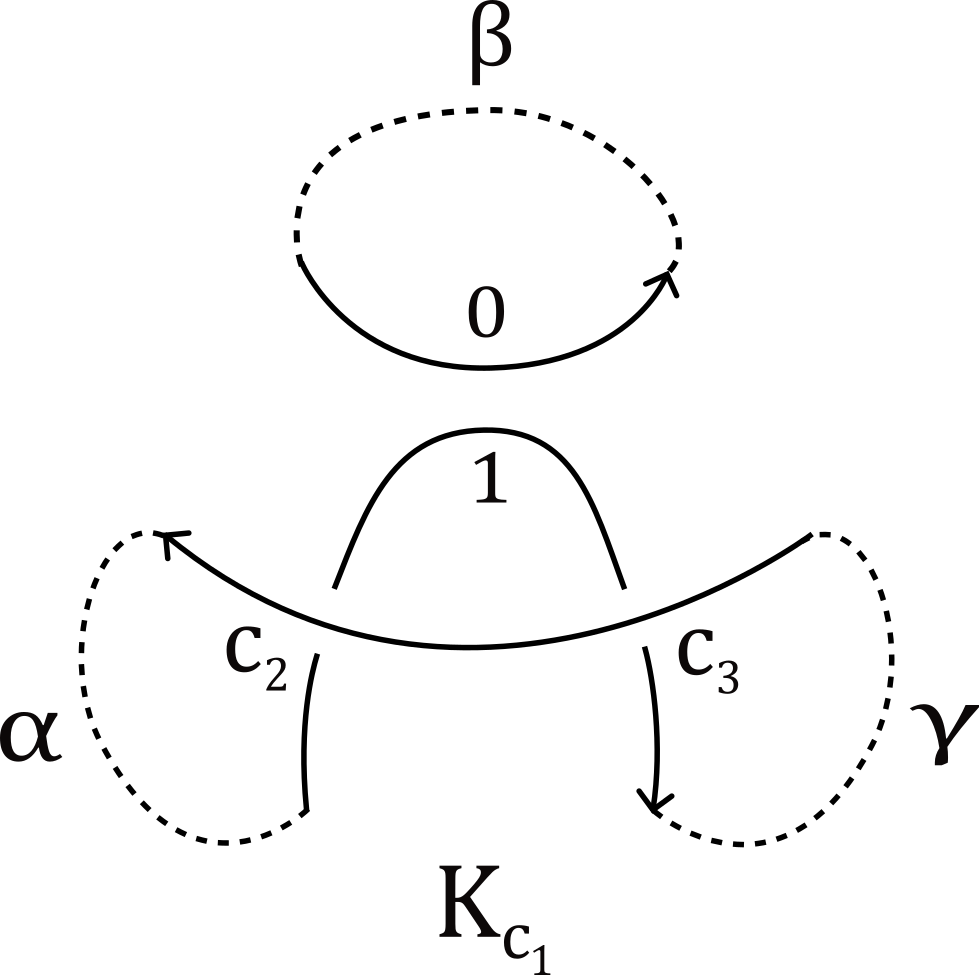}}
		\qquad 
		\subfloat{\includegraphics[width=0.25\textwidth]{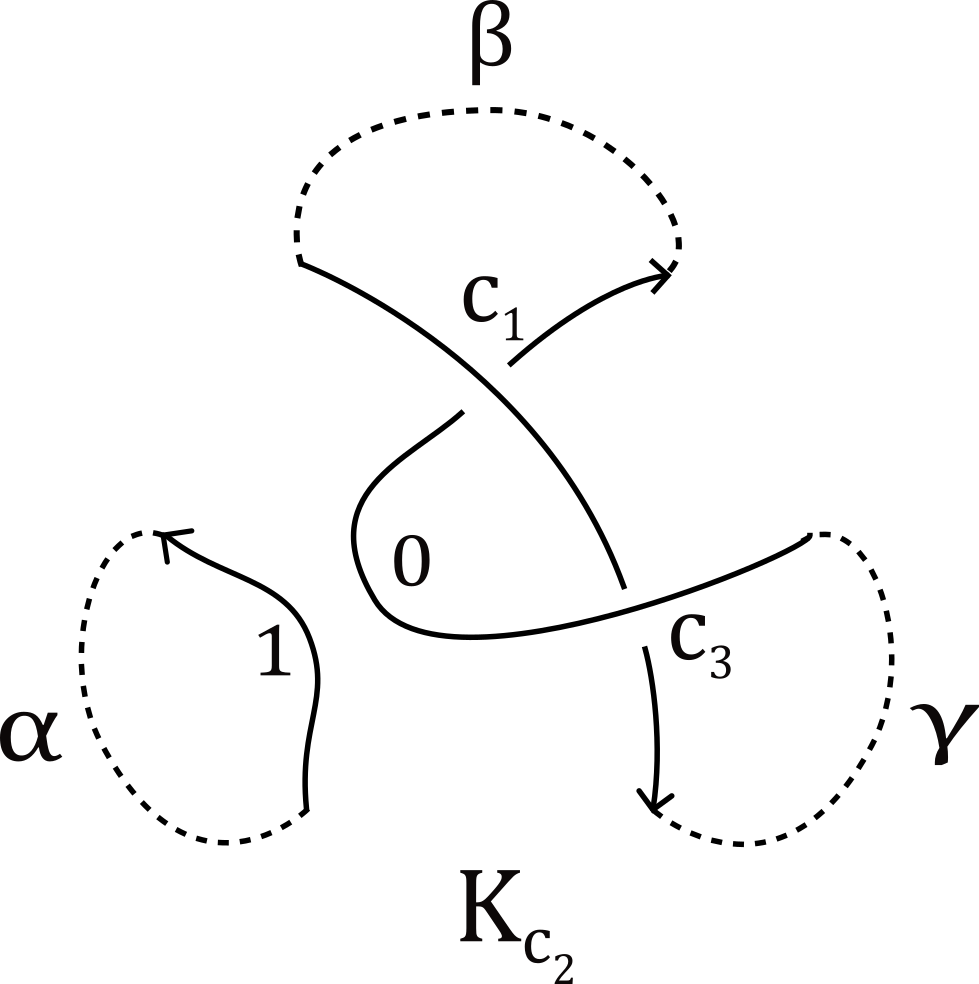}}
		\qquad 
		\subfloat{\includegraphics[width=0.25\textwidth]{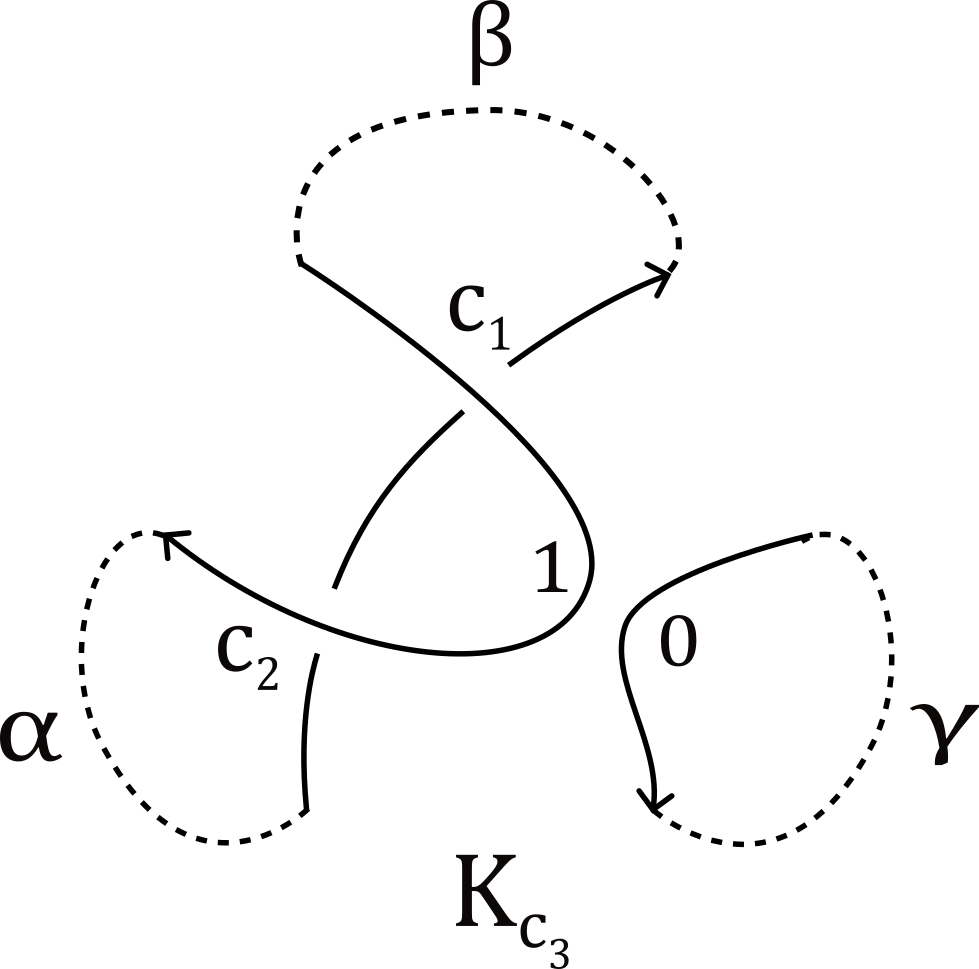}}
		\caption{Smoothing of $K$ at $c_{1}, c_{2}$, and  $c_{3}$ for the sequence $(S_{1}S_{2} S_{3})$.}
		\label{fig:Kc123}
	\end{figure}	
	\begin{figure}[htbp]
		\centering 
		\subfloat{\includegraphics[width=0.25\textwidth]{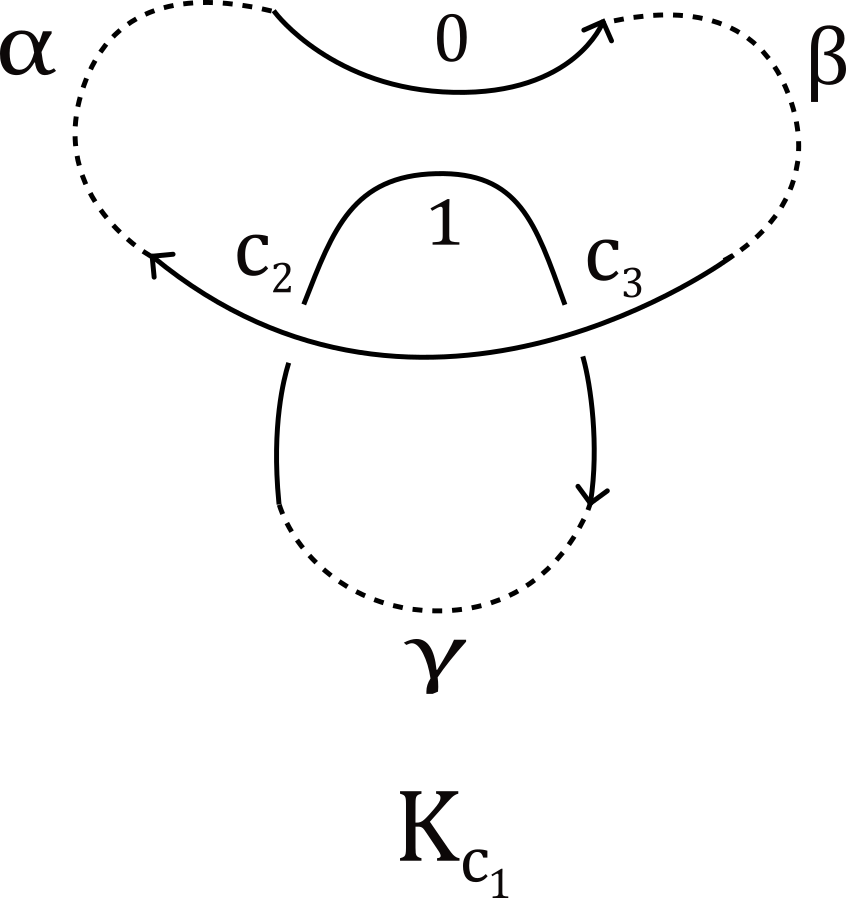}}
		\qquad 
		\subfloat{\includegraphics[width=0.25\textwidth]{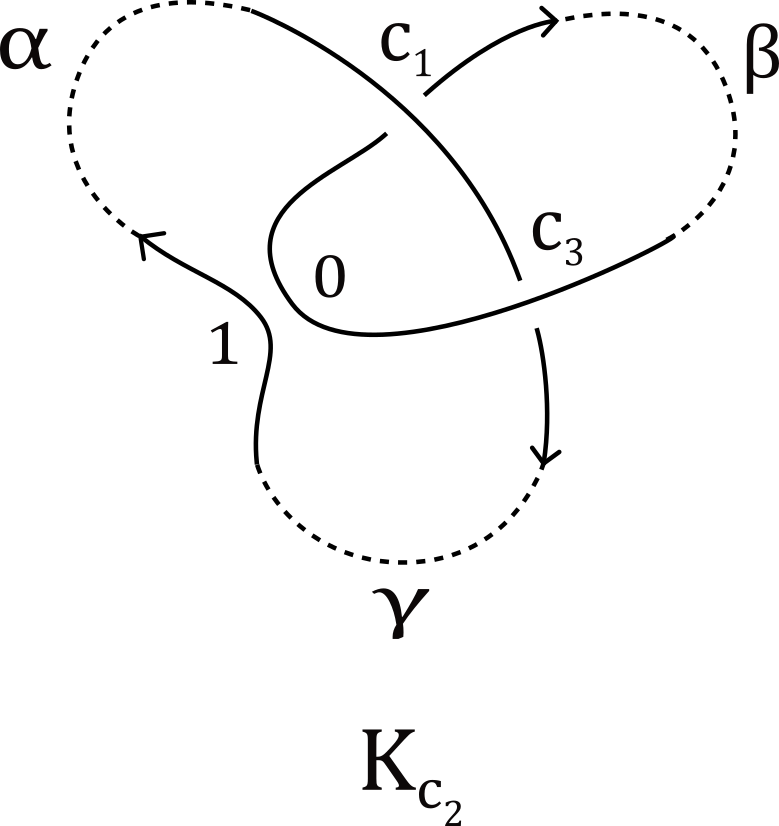}}
		\qquad 
		\subfloat{\includegraphics[width=0.25\textwidth]{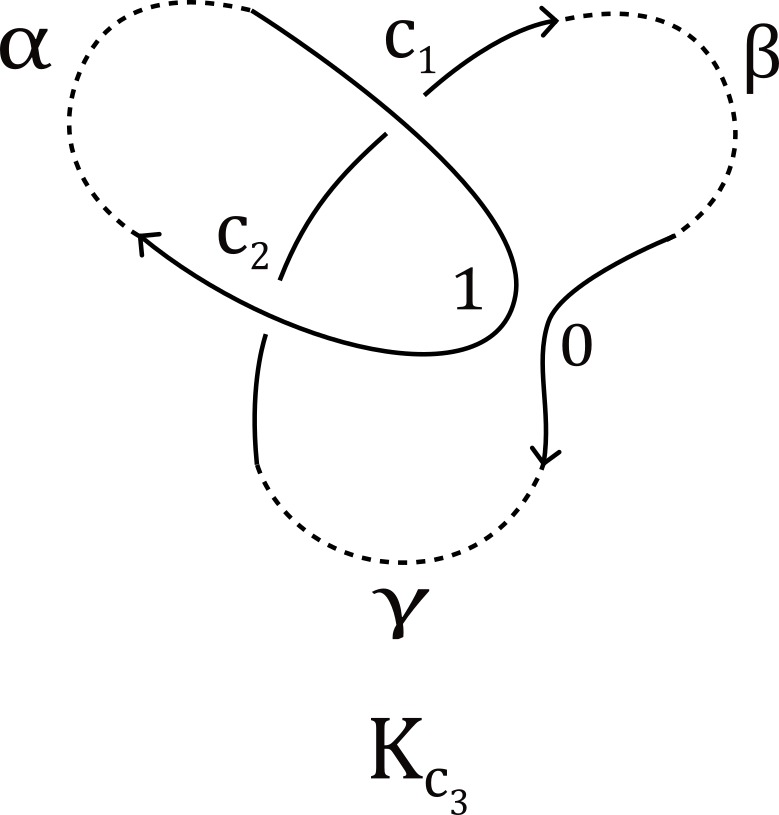}}
		\caption{Smoothing of $K$ at $c_{1}, c_{2}$, and$ c_{3}$ for the sequence $(S_{1}S_{3} S_{2})$.}
		\label{fig:Kc132}
	\end{figure}

	For distinct $\delta_1,\delta_2 \in \{\alpha,\beta,\gamma\}$,
	define $C(\delta_{1}\delta_{2})=$ crossings where $\delta_{1}$ passes over $\delta_{2}$. Also define
	\[
	l({\delta_1\delta_2})
	=\sum_{c \in C(\delta_{1}\delta_{2})} sgn(c).\]	
	
	From the diagrams of  $K_{c_1},K_{c_2}$, and $K_{c_3}$ for the sequence $(S_{1}S_{2} S_{3})$, we can write 
	\[ind^{\,0}(c_{1},K)= l({\alpha \beta})-l({\beta\alpha}) + l({\gamma\beta})- l({\beta \gamma}),\]
	\[ind^{\,0}(c_{2},K)= l({\alpha\beta})- l({\beta \alpha})+ l({\alpha\gamma})- l({\gamma\alpha}),\]
	\[ind^{\,0}(c_{3},K)= l({\beta\gamma})- l({\gamma \beta})+ l({\alpha\gamma})- l({ \gamma \alpha}),\]
	
	And, from the diagrams of  $K_{c_1},K_{c_2}$, and $K_{c_3}$ for the sequence $(S_{1}S_{3} S_{2})$, we can write 
	\[ind^{\,0}(c_{1},K)= l({\gamma \beta})-l({\beta\gamma}) + l({\gamma\alpha})- l({\alpha\gamma}),\]
	\[ind^{\,0}(c_{2},K)= l({\alpha\beta})- l({\beta \alpha})+ l({\gamma\beta})- l({\beta\gamma}),\]
	\[ind^{\,0}(c_{3},K
	)= l({\alpha\beta})- l({\beta\alpha})+ l({\alpha\gamma})- l({ \gamma \alpha}.\]
	For both sequences, we can write,
	\[
	ind^{\,0}(c_{2},K)
	=
	ind^{\,0}(c_{1},K)
	+
	ind^{\,0}(c_{3},K).
	\]
	Therefore,
	if two of the crossings belong to $ZC(K)$,
	then so does the third.
\end{proof}

The following lemma can be proved similarly.
\begin{lemma}\label{lem:r32}
	
	Let $\{i,j,k\}=\{1,2,3\}$. Consider the local diagram on the right hand side of the $C3a$ move (Fig.~\ref{fig:3crossings}) for a twisted knot diagram $K'$. 
	If $c_i^{\prime},c_j^{\prime} \in ZC(K')$, then $c_k^{\prime} \in ZC(K')$.
\end{lemma}
	
\begin{lemma}\label{lem:r33}
Let $\{i,j,k\}=\{1,2,3\}$and let $K$ and $K'$ be two twisted knot diagrams which are related by $C3a$ move (Fig.~\ref{fig:3crossings}). 
	If $c_i \in ZC(K)$, then $c_i^{\prime} \in ZC(K')$. And $z(c_i)= z(c'_i)$ for $i=1,2,3$.
	Moreover, $p^{l}(c_i)= p^{l}(c'_i),\, i=1,2,3; l \in \{0,1\}$.
\end{lemma}
\begin{proof}
	From Fig.~\ref{fig:move3a} we see that the components of $K_{c_{i}}$ and $K'_{c_{i}^{\prime}}\,,i \in \{1,2,3\}$ only differ locally by two crossings with the third strand labeled $l_1\, (\in \{0,1\})$. 

\begin{figure}[htbp]
	\centering 
	\subfloat{\includegraphics[width=0.4\textwidth]{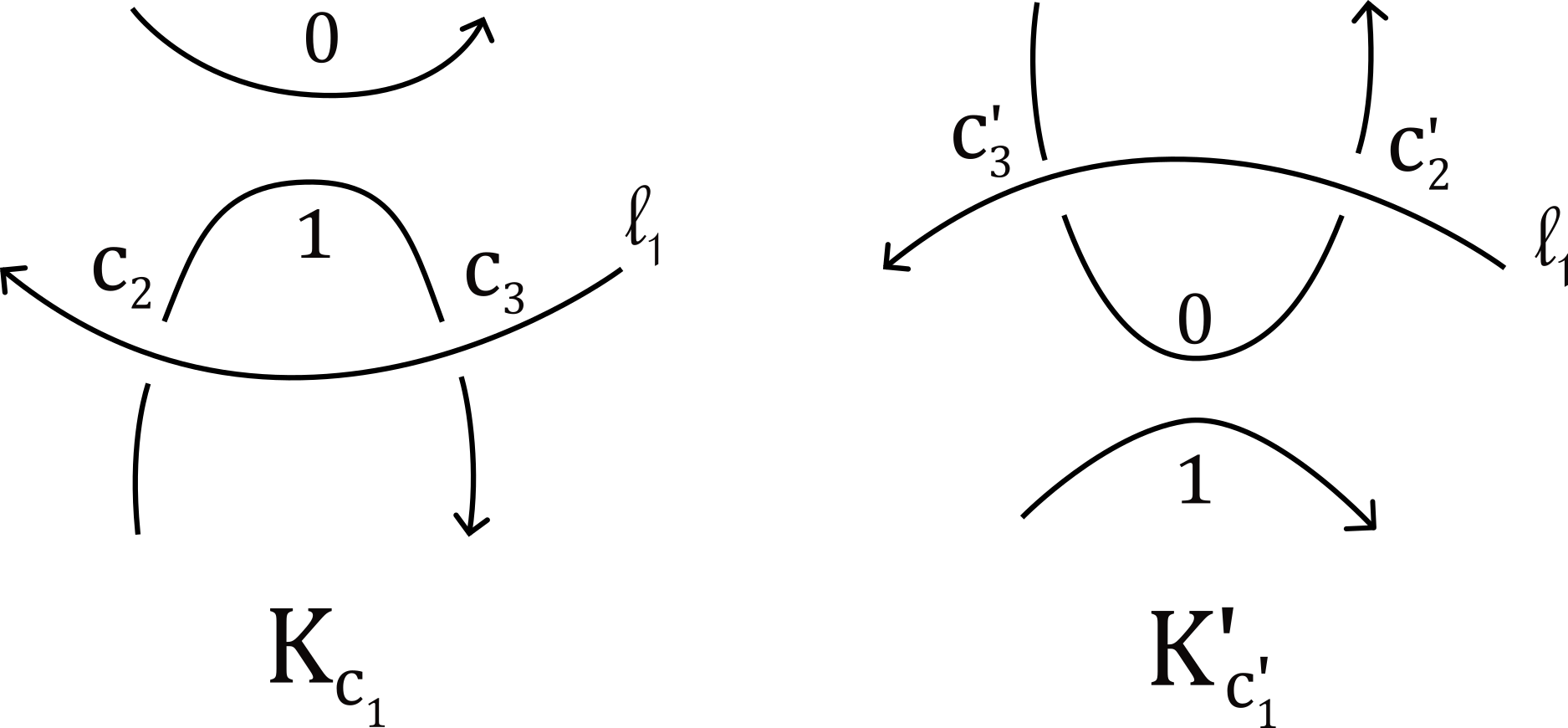}}
	\qquad 
	\subfloat{\includegraphics[width=0.4\textwidth]{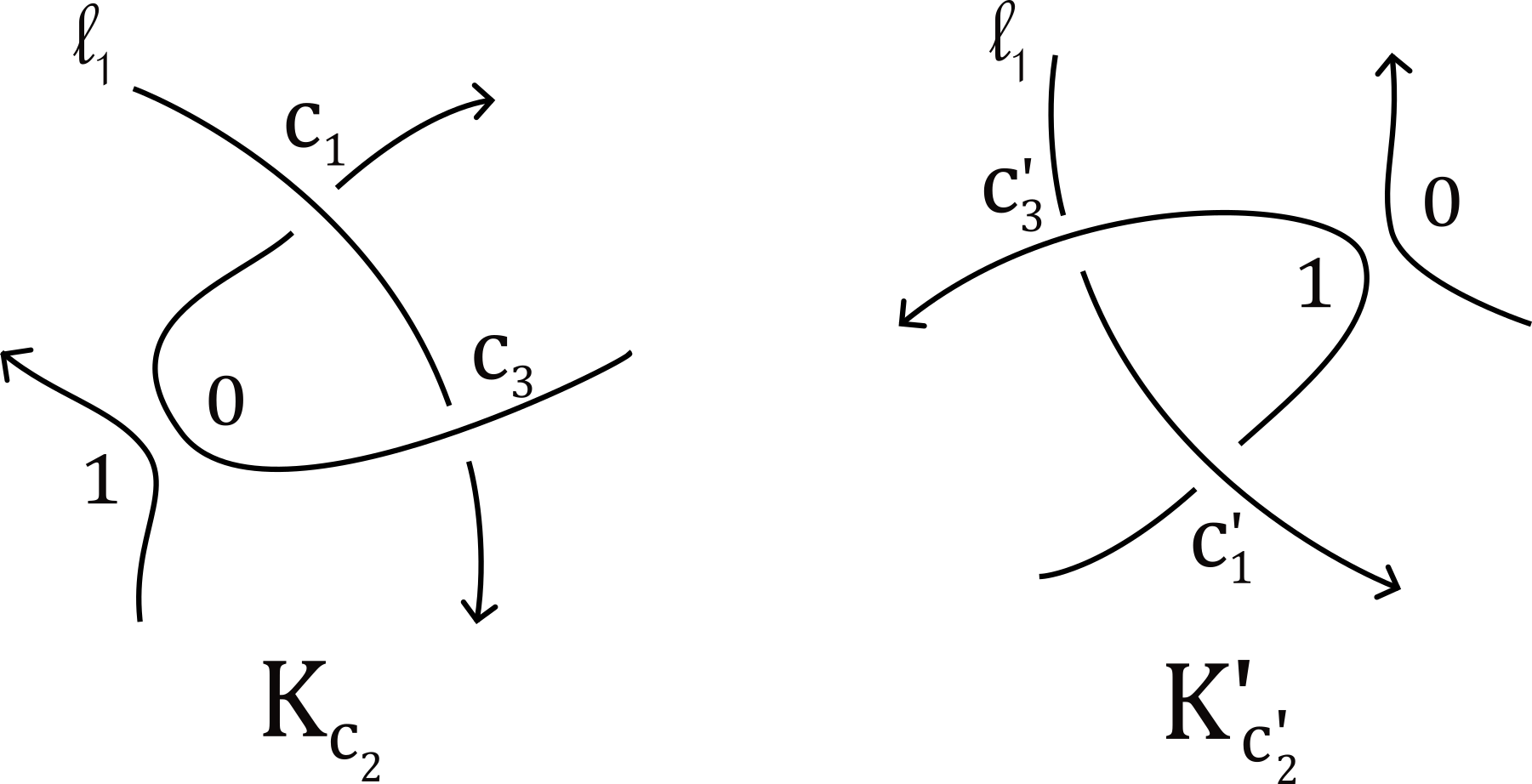}}
	\newline
	\vfill
	\subfloat{\includegraphics[width=0.4\textwidth]{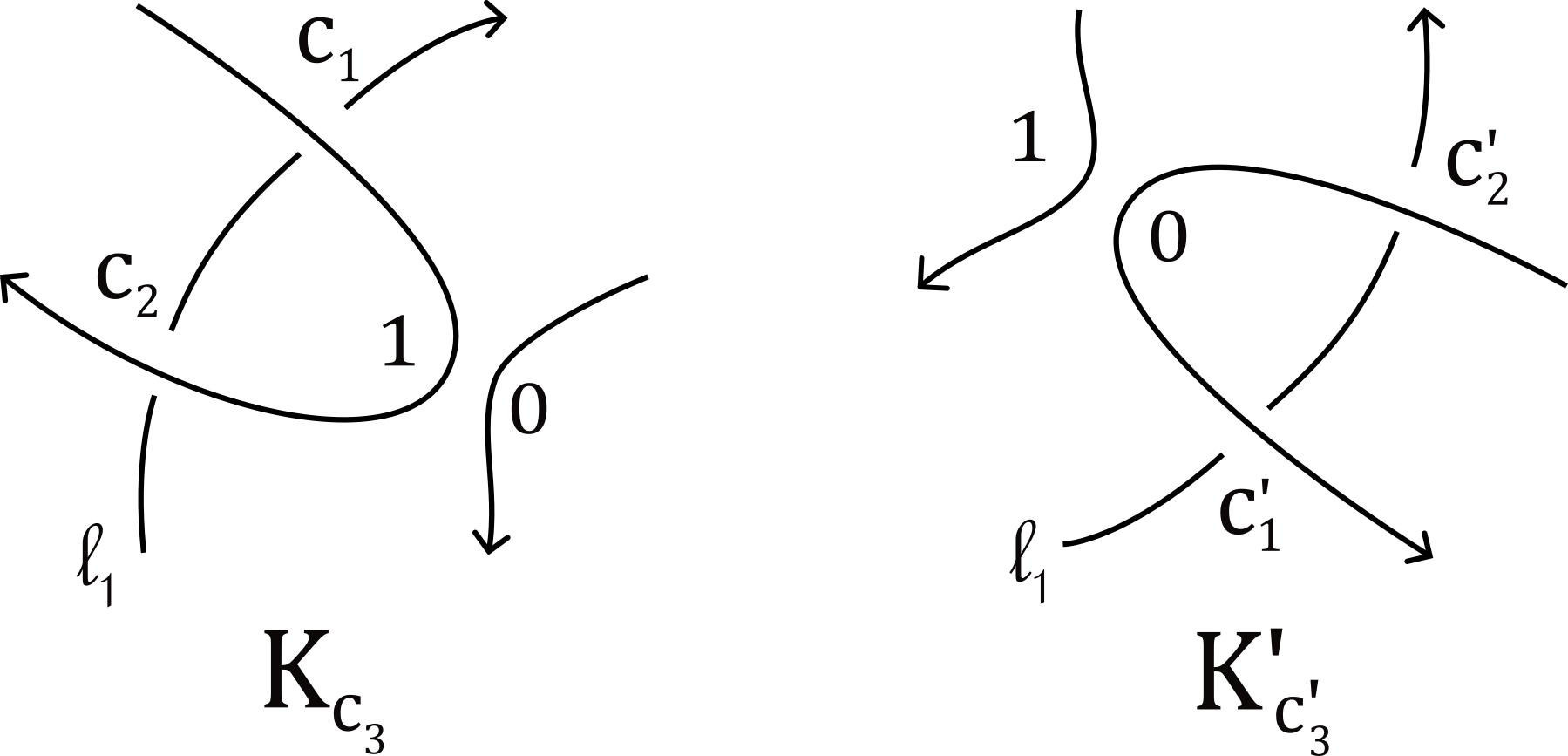}}
	\caption{Smoothing diagrams of $K$ at $c_i$ and $K'$ at $c_i'$, for $i=1,2,3$.}
	\label{fig:move3a}
\end{figure}

	Now, the strands that are involved with those two crossings can be from the same component or different components. If both are from the same component then they are self-crossings. And, if they are from different components then their signs
	cancel each other in every case. Therefore, \[ind^{\,l}(c_{i},K)= ind^{\,l}(c_{i}^{\prime},K')\,,i \in \{1,2,3\};\, l \in \{0,1\},\]
	which implies $
	c_{i}\in ZC(K)$ if and only if $c_{i}^{\prime}\in ZC(K')\,,i \in \{1,2,3\}$. 
	Assume that $c_i \in ZC(K)$. Then by Lemma~\ref{lem:r31}, $c_j \in ZC(K)$ if and only if $c_k \in ZC(K)$ for $i \ne j \ne k$. Similarly $c'_j \in ZC(K')$ if and only if $c'_k \in ZC(K')$ for $i \ne j \ne k$ by Lemma~\ref{lem:r32}. Since  $c_j, c_k$ (resp. $c'_j, c'_k$) either have the same sign but opposite over/under information or opposite signs but the same over/under information, we have $z(c_i)=z(c'_i)$ for $i=1,2,3$.

	Since  there are no bars involved in the moves, and the labels on the components are preserved, we have $p^{l}(c_i)= p^{l}(c'_i),\, i=1,2,3$, for $l \in \{0,1\}$.
\end{proof}
\begin{lemma}\label{lem:c3inv}
	The polynomial $Q^{z}_{K}(s,t)$ is invariant under the $C3a$-move.
\end{lemma}

\begin{proof}
	Let the twisted knot diagrams $K$ and $K'$ are related by the $C3a$ move. 
	For all the crossings $e \in C(K) \setminus \{c_{1}, c_{2}, c_{3}\}= C(K') \setminus \{c_{1}^{\prime}, c_{2}^{\prime}, c_{3}^{\prime}\}$, $K_{e}$ and $K'_{e}$ are related as shown in Fig.~\ref{fig:3labels}. Let $l_{1},l_{2},l_{3}$ be the component labels on the three strands. 
	\begin{figure}[htbp]
		\centering
		\includegraphics[width=0.45\textwidth]{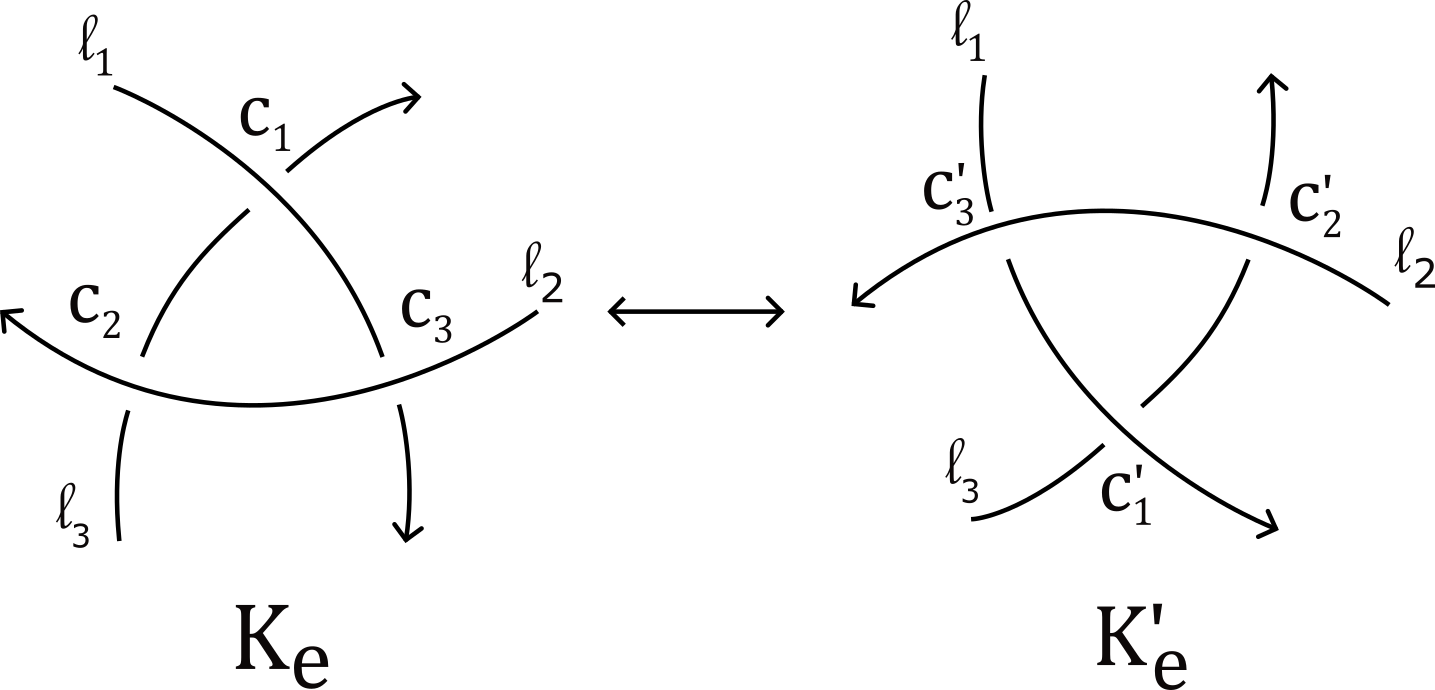}
		\caption{Component labels $l_1, l_2, l_3$ on the three strands involved in the $C3a$ move.}
		\label{fig:3labels}
	\end{figure}

	Case 1: When $l_{1}=l_{2}=l_{3}$ then all crossings $c_{i}$ and $c_{i}^{\prime}\,,i \in \{1,2,3\}$ are self-crossings. Therefore, \[ind^{\,l}(e,K)= ind^{\,l}(e,K')\quad, l \in \{0,1\}.\]
	
	Case 2: When $l_{i}=l_{j}\neq l_{k}\,; i,j,k \in \{1,2,3\}$, then 
	the pair of crossing between the components in the local diagram of $K_{e}$ and $K'_{e}$
	appears as both overcrossings or both undercrossings with  opposite signs or one overcrossing and one undercrossing with same sign. Therefore, \[ind^{\,l}(e,K)= ind^{\,l}(e,K'),\quad l \in \{0,1\}.\]
	Both cases imply that $e \in ZC(K)$ if and only if  $e \in ZC(K').$
	
	By Lemmas~\ref{lem:r31}, \ref{lem:r32},and  \ref{lem:r33}, we have $zind^{\,l}(e,K)=zind^{\,l}(e,K'),\, l \in \{0,1\}$ when $e \in ZC(K)$. Then, $z(e)$ do not change under the move. Moreover, $p^{l}(e),\, l \in \{0,1\}$ also do not change as the local diagram of $C3a$ does not involve any bars.
	Hence, 	by Lemmas~\ref{lem:r31}, \ref{lem:r32}, and \ref{lem:r33}, we have, $Q^{z}_{K}(s,t=Q^{z}_{K'}(s,t)$.
	This completes the proof.
\end{proof}

\begin{lemma}\label{lem:vinv}
	$Q^{z}_{K}(s,t)$ is invariant under all virtual Reidemeister moves.
\end{lemma}
\begin{proof}
		The moves $V1a$, $V2a$, and $V3a$ involve only virtual crossings. And in the move $V4g$, a strand passes through a virtual crossing. This changes only the position of the virtual crossing and leaves all classical crossings, their signs, component labels, affine indices, and the parity of the number of bars on each component unchanged. Therefore, $Q^{z}_{K}(s,t)$ is invariant under all virtual Reidemeister moves.
\end{proof}
Now, we show that $Q^{z}_{K}(s,t)$ is invariant under the twisted moves $T1a, T2a, T3a$ and $T3c$. Since  the $T1a$ and $T2a$ move do not involve any classical crossings, it is enough to show that $Q^{z}_{K}(s,t)$ is invariant under the twisted moves $T3a$ and $T3c$.
\begin{lemma}\label{lem:t3inv}
	$Q^{z}_{K}(s,t)$ is invariant under the twisted moves $T3a$ and $T3c$.
\end{lemma}
\begin{proof}
	Suppose the twisted knot diagrams $K$ and $K'$ are related by the $T3a$ move as shown in Fig.~\ref{fig:t3fig1}.
		\begin{figure}[htbp]
		\centering
		\includegraphics[width=0.7\textwidth]{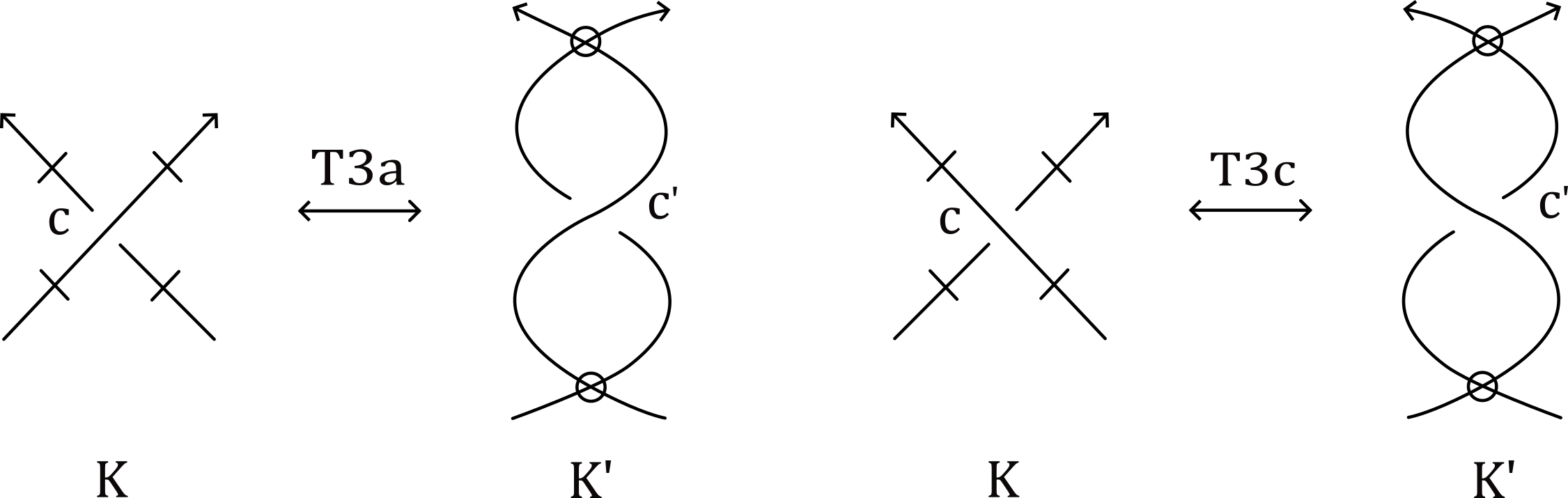}
		\caption{The twisted knot diagrams $K$ and $K'$ related by twisted Reidemeister moves $T3a$ and $T3c$.}
		\label{fig:t3fig1}
	\end{figure}
		For $e \in C(K) \setminus \{c\} = C(K') \setminus \{c^{\prime}\}$, $K_{e}$ and $K'_{e}$ are related as shown in Fig.~\ref{fig:t3e}. Let $l_{1}$, $l_{2}$ ($\in \{0,1\}$) be the component labels on the strands. Notice that $sgn(c)= sgn(c^{\prime})$, but the crossing type changes from over (resp. under) to under (resp. over) in each component. \\
	Therefore, when $l_{1} \neq l_{2}$,
	\[|ind^{\,l}(e, K)-ind^{\,l}(e, K')|=|sgn(c)+sgn(c')|=2, \quad \text{for} \,\; l \in \{0,1\}.\]
	And if $l_{1}=l_{2}$ then $c$ and $c'$ are self-crossings.
	Therefore, in both cases $e \in ZC(K)$ if and only if $e \in ZC(K')$. 
	\begin{figure}[htbp]
	\centering
		\subfloat[\label{fig:t3e}]{\includegraphics[width=0.34\textwidth]{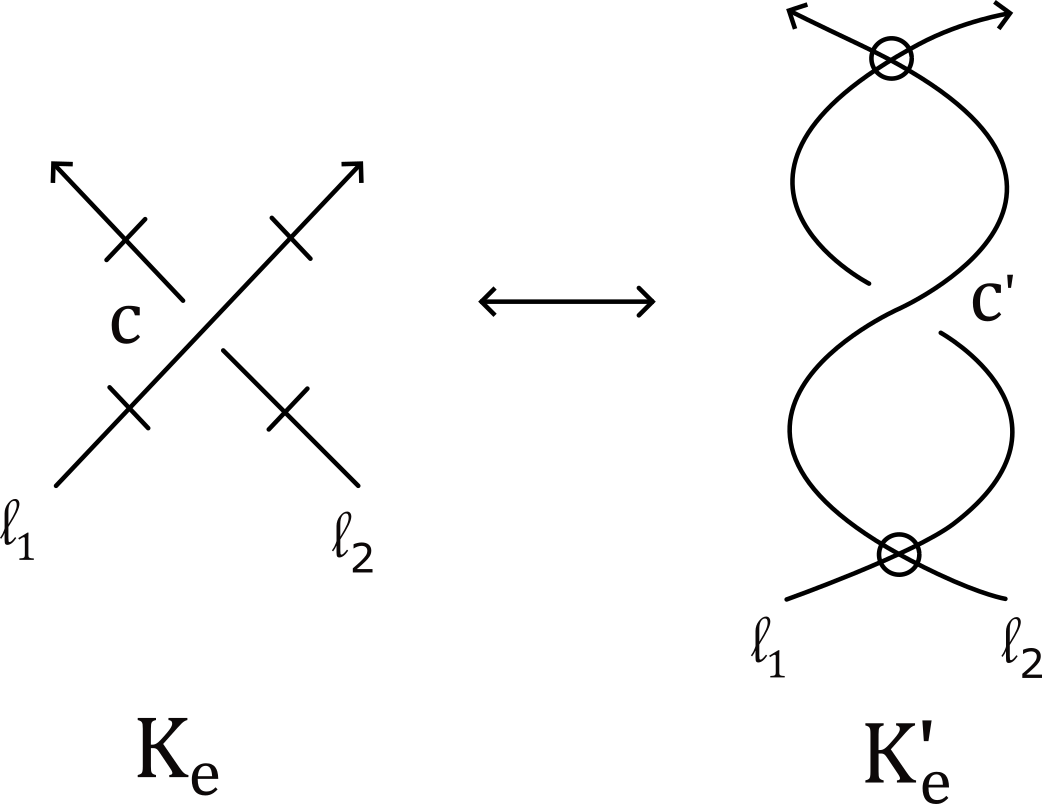}}
			\hspace{1.5cm}
	\subfloat[\label{fig:t3c}]{\includegraphics[width=0.31\textwidth]{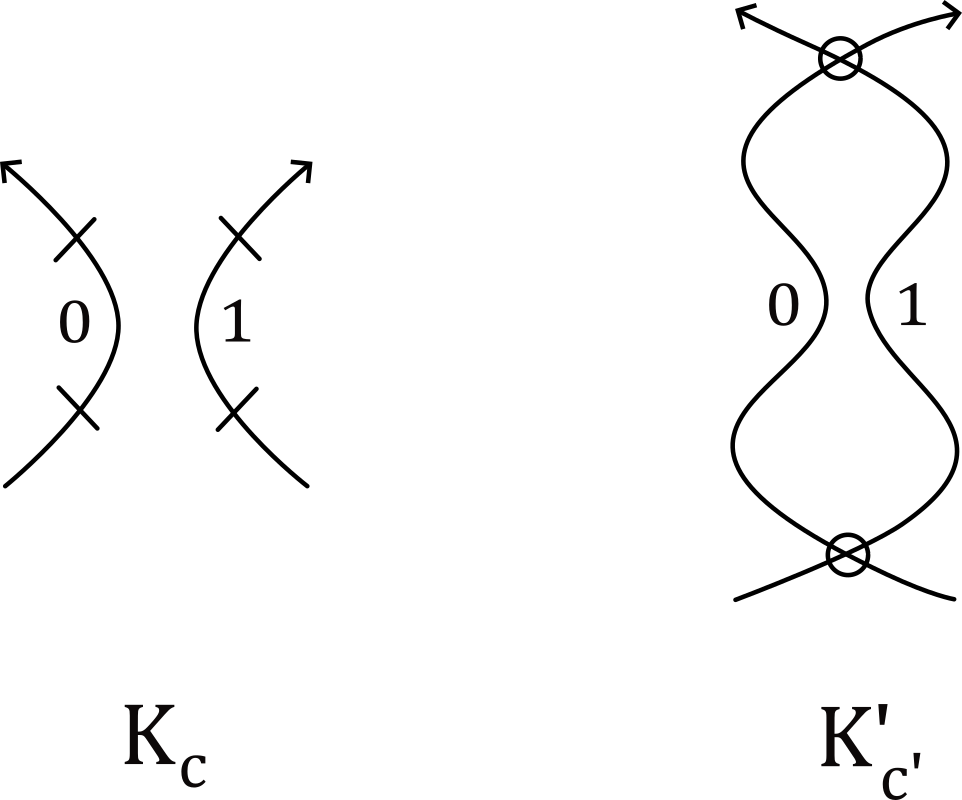}}
\caption{Smoothing diagrams of $K$ and $K'$. }
	\label{fig:t3fig2}
\end{figure}
	We can see from Fig.~\ref{fig:t3c} that the components of $K_{c}$ and $K'_{c}$  are the same outside the region where the $T3a$ move is applied. Then, $$ind^{\,0}(c,K)=ind^{\,1}(c^{\prime},K'),\quad ind^{\,1}(c,K)=ind^{\,0}(c^{\prime},K').$$

Since  $K$ and $K'$ are twisted knot diagrams,  $$ind^{\,0}(c,K)=-ind^{\,1}(c,K),\quad ind^{\,0}(c',K')=-ind^{\,1}(c',K').$$ 
	 Then, $ind^{\,0}(c,K)=-ind^{\,0}(c',K')$
	which implies $c \in ZC(K)$ if and only if $c' \in ZC(K')$.
	Therefore, for $e \in ZC(K)$, $z(e)$ do not change. Also for all $e \in C(K) \setminus \{c\} = C(K') \setminus \{c^{\prime}\}$, $p^{l}(e)$, $l \in \{0,1\})$ do not change under $T3a$ move.
	
	 For $c \in ZC(K)$, $z(c) = z(c')$ and
	\[ p^{0}(c) = p^{1}(c'),\quad p^{1}(c) = p^{0}(c'). \]
Since  $sgn(c)= sgn(c')$, we have $Q^{z}_{K}(s,t)= Q^{z}_{K'}(s,t)$.

Similarly, we can show that $Q^{z}_{K}(s,t)$ is invariant under $T3c$ move.
\end{proof}
 By Lemmas~\ref{lem:c1inv}, \ref{lem:c2inv}, \ref{lem:c3inv}, \ref{lem:vinv}, and \ref{lem:t3inv}, we can conclude the following result.
\begin{theorem}
	$Q^{z}_{K}(s,t)$ is invariant under all classical, virtual and twisted Reidemeister moves.
\end{theorem}
\section{Examples}\label{sec:examples}
In this section, we provide some examples of twisted knots that cannot be distinguished by the $Q$-polynomial but are distinguishable by the $Q^{z}$-polynomial.

	\begin{figure}[htbp]
	\centering
	\subfloat{\includegraphics[width=0.2\textwidth]{ex1}}
	\hspace{1.5cm}
	\subfloat{\includegraphics[width=0.14\textwidth]{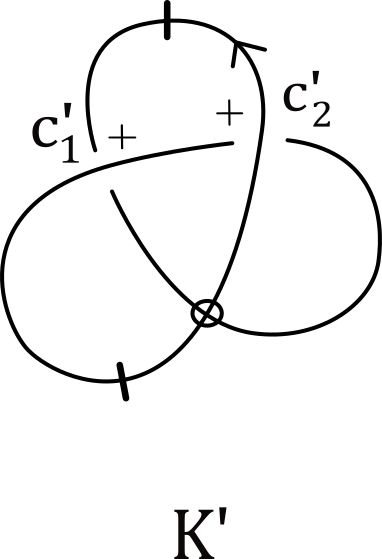}}
	\caption{}
	\label{fig:ex2}
\end{figure}

\begin{example}\label{ex:vasdia}
	For both twisted knot diagrams $K$ and $K'$ in Fig.~\ref{fig:ex2}, $Q(s,t)= (s-1)^{2}+(st-1)^{2}.$
	 But, for $K$ (See  Example~\ref{ex:ex1}), we have
	$$Q_{K}^{z}(s,t)=3(s-1)^{2}+(st-1)^{2},$$
	and for $K'$ (See Table~\ref{tab:ex2K}), we have
	$$Q_{K'}^{z}(s,t)=(s-1)^{2}+(st-1)^{2}.$$
		\begin{table}[htbp]
			\centering
			\small
			\begin{tabular}{|c|c|c|c|c|c|} 
				\hline
				$c \in C(K')$ & $sgn(c)$ & $ind^{\,0}(c,K')$  & $p^{0}(c)$ & $p^{1}(c)$ \\ 
				\hline 
				$c'_{1}$ & $+1$ & $1$ & $0$ & $0$ \\ 
				\hline
				$c'_{2}$ & $+1$ & $-1$ &$1$ & $1$ \\
				\hline
			\end{tabular}
				\caption{The values of $ind^{\,0}(c,K'), p^{0}(c), p^{1}(c)$ for all crossings of $K'$.}
			\label{tab:ex2K}
		\end{table}

\end{example}
	\begin{figure}[htbp]
	\centering 
	\includegraphics[width=\textwidth]{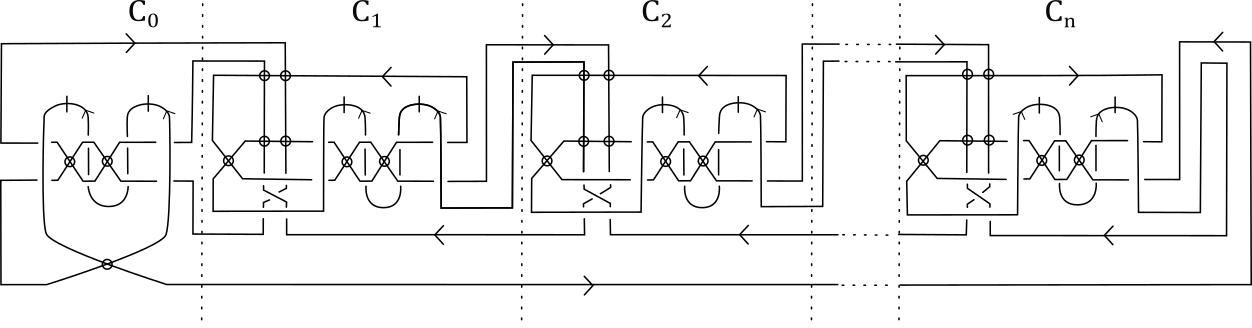}
	\caption{The twisted knot diagram $T_n$ for general $n$.}
	\label{fig:reggor1}
\end{figure}
\begin{example}
Here we provide family of twisted knot diagrams $\{T_n,\, n \in \mathbb{N}\}$ which cannot be distinguished by the $Q$-polynomial. Now, we compute the $Q^{z}$-polynomial of $T_n$ for each $n \in \mathbb{N}$ 
and prove that it can distinguish the knots in the family $T_n,\, n \in \mathbb{N}$  (See Fig.~\ref{fig:reggor1},\ref{fig:regblock}).

The $Q$-polynomial is $(st-1)^{2}-(t-1)^{2}$ for all $n \in \mathbb{N}$.
But,
 for $n \in  \mathbb{N}$,
 \[Q^{z}_{T_n}(s,t)= (n+2)(st-1)^{2}+(n+3)(s-1)^{2}-(n+2)(t-1)^{2}.
\]
See Table~\ref{tab:nonzero}, Table~\ref{tab:zero}.

\begin{figure}[htbp]
	\centering 
	\includegraphics[width=\textwidth]{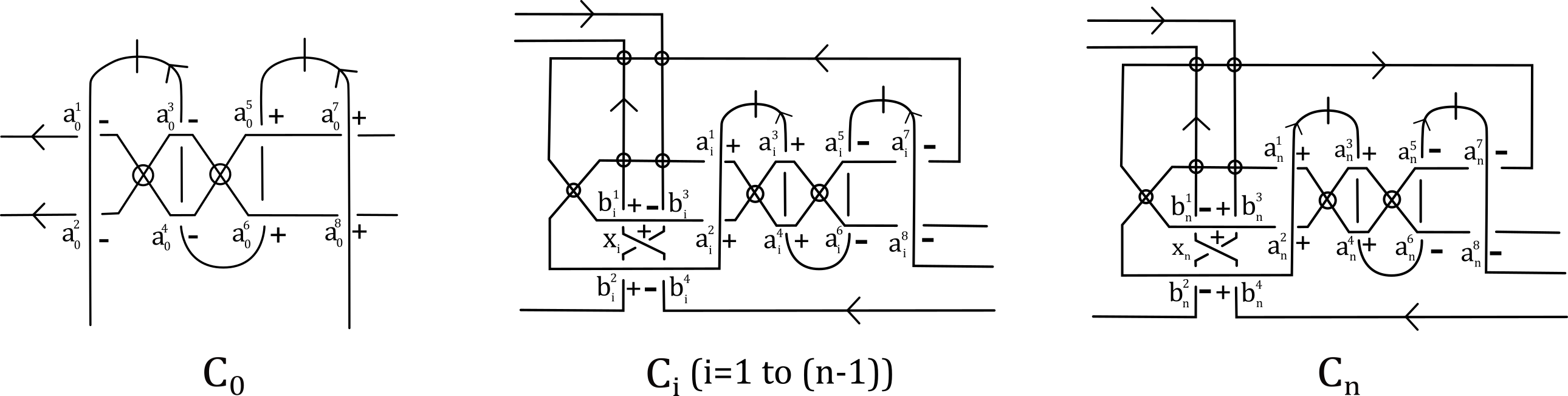}
	\caption{The repeating building blocks used to construct the twisted knot diagrams $T_n$, $n \in \mathbb{N}$.}
	\label{fig:regblock}
\end{figure}
\begin{table}[H]
	\centering
	\resizebox{0.5\textwidth}{!}{
	\begin{tabular}{ |c|c|c|c|} 
		\hline
		crossings $c$ & $sgn(c)$ & $p^{0}(c)$ & $p^{1
		}(c)$ \\  
		\hline
		$a_{0}^{1}$ & $-1$   & $0$ & $0$ \\ 
		$a_{0}^{2}$ & $-1$  & $0$ & $0$ \\
		$a_{0}^{5}$ &$+1$ &  $1$ & $1$ \\         
		$a_{0}^{7}$ & $+1$ & $0$ & $0$ \\
		
		$a_{0}^{8}$ & $+1$  & $0$ & $0$ \\
		\hline
		$a_{i}^{1},\, i=1,\cdots,(n-1)$ & $+1$  & $0$ & $0$ \\ 
		$a_{i}^{3},\, i=1,\cdots,(n-1)$ & $+1$ &  $1$ & $1$ \\
		$a_{i}^{5},\, i=1,\cdots,(n-1)$ &$-1$ & $1$ & $1$ \\
		$a_{i}^{7},\, i=1,\cdots,(n-1)$ & $-1$ & $0$ & $0$ \\
		\hline
		$a_{n}^{1}$ & $+1$ & $0$ & $0$ \\ 
		$a_{n}^{3}$ & $+1$  & $1$ & $1$ \\
		$a_{n}^{5}$ &$-1$  & $1$ & $1$ \\
		$a_{n}^{7}$ & $-1$ & $0$ & $0$ \\
		\hline
	\end{tabular}}
	\caption{The values of $p^{0}(c), p^{1}(c)$ for all crossings with odd affine index in $T_n,\, n \in \mathbb{N}$.}
	\label{tab:nonzero}
\end{table}

\begin{longtable}{|c|c|c|c|c|c|}
	\hline
	crossings $c$ & $sgn(c)$ & $zind^{\,0}(c)$ & $z(c)$ & $p^{0}(c)$ & $p^{1}(c)$ \\ 
	\hline
	\endfirsthead
	
	\hline
	crossings $c$ & $sgn(c)$ & $zind^{\,0}(c)$ & $z^{0}(c)$ & $p^{0}(c)$ & $p^{1}(c)$ \\
	\hline
	\endhead
	
	\hline
	\multicolumn{6}{r}{{Continued...}} \\
	\endfoot
	
	\hline
	\addlinespace
	\caption{$z(c), p^{0}(c), p^{1}(c)$ for all zero crossings in $T_n,\, n \in \mathbb{N}$.} \label{tab:zero}
	\endlastfoot
	
	$a_{0}^{3}$ & $-1$ & $0$ & $0$ & $1$ & $1$ \\ 
	$a_{0}^{4}$ & $-1$ & $0$ & $0$ &$1$ & $1$ \\
	$a_{0}^{6}$ &$+1$ & $-1$&  $1$ & $1$ & $1$ \\ 
	\hline
	$a_{i}^{2},\, i=1,\cdots,(n-1)$ & $+1$ & $+1$ & $1$ & $0$ & $0$ \\ 
	$a_{i}^{4},\, i=1,\cdots,(n-1)$ & $+1$ & $-1$ & $1$ & $1$ & $1$ \\
	$a_{i}^{6},\, i=1,\cdots,(n-1)$ &$-1$ & $0$&  $0$ & $1$ & $1$ \\
	$a_{i}^{8},\, i=1,\cdots,(n-1)$ & $-1$ &  $0$ & $0$ & $0$ & $0$ \\ 
	\hline
	$a_{n}^{2}$ & $+1$ & $-1$ & $1$ & $0$ & $0$ \\ 
	$a_{n}^{4}$ & $+1$ & $+1$ & $1$ &$1$ & $1$ \\
	$a_{n}^{6}$ &$-1$ & $0$&  $0$ & $1$ & $1$ \\
	$a_{n}^{8}$ &$-1$ & $0$&  $0$ & $0$ & $0$ \\ 
	\hline
	$b_{1}^{1}$ & $+1$ & $-1$ & $1$ & $0$ & $0$ \\
	$b_{1}^{2}$ & $+1$ & $-1$ & $1$ & $0$ & $0$ \\
	$b_{1}^{3}$ & $-1$ & $-2$ & $0$ &  $0$ & $0$ \\
	$b_{1}^{4}$ & $-1$ & $0$ & $0$ &  $0$ & $0$ \\ 
	\hline
	$b_{i}^{1},\, i=2,\cdots,(n-1)$ & $+1$ & $-1$ & $1$ & $0$ & $0$ \\ 
	$b_{i}^{2},\, i=2,\cdots,(n-1)$ & $+1$ & $0$ & $0$ & $0$ & $0$ \\
	$b_{i}^{3},\, i=2,\cdots,(n-1)$ &$-1$ & $-1$&  $1$ & $0$ & $0$ \\
	$b_{i}^{4},\, i=2,\cdots,(n-1)$ & $-1$ &  $0$ & $0$ & $0$ & $0$ \\ 
	\hline
	$b_{n}^{1}$ & $-1$ & $+1$ & $1$ & $0$ & $0$ \\
	$b_{n}^{2}$ & $-1$ & $0$ & $0$ & $0$ & $0$ \\
	$b_{n}^{3}$ & $+1$ & $+1$ & $1$ &  $0$ & $0$ \\
	$b_{n}^{4}$ & $+1$ & $0$ & $0$ &  $0$ & $0$ \\
	\hline
	
	$x_{1}$ &$+1$ & $-1$&  $1$ & $0$ & $0$ \\
	$x_{i}\,, i=2,\cdots,n-1$ & $+1$ &  $0$ & $0$ & $0$ & $0$ \\
\end{longtable}
\end{example}

\section{Properties of the  \texorpdfstring{$Q^{z}$}{Qz}-polynomial}\label{sec:operations}

\begin{theorem}
	Let $K$ be a twisted knot diagram. Then
	\begin{enumerate}
		\item If $-K$ is the inverse of $K$ then \[Q^{z}_{-K}(s,t)=Q^{z}_{K}(s,t).\]
		\item If $K^{*}$ is the mirror image of $K$ then \[Q^{z}_{K^{*}}(s,t)=-Q^{z}_{K}(s,t).\]
	\end{enumerate}
\end{theorem}
\begin{proof}
\begin{enumerate}
	\item Let $c'$ be the crossing of $-K$ corresponding to the crossing $c$ in $K$. Note that the component labels of $K_{c}$ are switched in $-K_{c'}$. Then for all $c \in K$, $$ind^{\,0}(c,K)=ind^{\,1}(c',-K)= -ind^{\,0}(c',-K)$$ which implies $c \in ZC(K)$ if and only if $c' \in ZC(-K)$. Then
	 \[z^{\,0}(c)=z^{\,1}(c'), \quad z^{\,1}(c)=z^{\,0}(c')\quad \text{when} \,\, c \in ZC(K)  .\]

	Moreover, $$p^{0}(c)=p^{1}(c')\quad \text{and}\, p^{1}(c)=p^{0}(c')$$ for all $c \in C(K)$ and $c' \in C(-K)$. Since  $sgn(c) = sgn(c')$, we have $$Q^{z}_{-K}(s,t)=Q^{z}_{K}(s,t).$$
	
	\item Let $c'$ be the crossing of $K^{*}$ corresponding to the crossing $c$ in $K$. Note that the crossing information at every crossing $c \in C(K)$ are switched from over (resp. under) to under (resp. over), as well as labels on the components are switched from $K_{c}$ to $K^{*}_{c'}$. Then for all $c \in C(K)$ and $c' \in C(K^{*})$, we have  $$ind^{\,0}(c,K)=-ind^{\,1}(c',K^{*})= -ind^{\,0}(c',K^{*})$$ which implies $c \in ZC(K)$ if and only if $c' \in ZC(K^{*})$. Then
\[z(c)=z(c') \text{when} \,
	c \in ZC(K).\]
	
	Moreover, $$p^{0}(c)=p^{1}(c')\quad \text{and}\, p^{1}(c)=p^{0}(c')$$ for all $c \in C(K)$ and $c' \in C(K^{*})$. Since  $sgn(c) = -sgn(c')$, we have  $$Q^{z}_{K^{*}}(s,t)=-Q^{z}_{K}(s,t).$$
\end{enumerate}
\end{proof}

\section{Applications of the \texorpdfstring{$Q^{z}$}{Qz}-polynomial}\label{sec:applications}
\subsection{Crossing change and Gordian distance}
Crossing change is an unknotting operation for classical knots i.e., any classical knot diagram can be transformed into a trivial knot diagram by a sequence of crossing changes. But the analogue statement for twisted knot diagrams may not hold. Some twisted knot diagrams cannot be transformed into a trivial twisted knot diagram by crossing change operation. Note that there are two trivial twisted knots, one without a bar and one with bar. Two twisted knot diagrams are said to be \textit{homotopic} if they can be related by a finite sequence of crossing changes and Reidemeister moves for twisted knots.

\begin{definition}
	Let $K$ and $K'$ be two homotopic twisted knots, then the {\it Gordian distance} denoted by $d_{G}(K,K')$ is the minimum number of crossing changes required to transform  a diagram of $K$ into a diagram of $K'$. 
\end{definition}

  In the following result we find the relation between the $Q^{z}$-polynomial for two twisted knots $K$ and $K'$ if they are related by one crossing change.

\begin{lemma}\label{lem:ccmove}
	Let $K$ and $K'$ be two twisted knot diagrams related by one crossing change move at a crossing $c$. When $c \notin ZC(K)$ we have,
	\[Q^{z}_{K}(s,t)-Q^{z}_{K'}(s,t) = 2\,sgn(c)(s\,t^{p^{0}(c)}-1)(s\,t^{p^{1}(c)}-1),\]
	and, when $c \in ZC(K)$ we have,
	\[Q^{z}_{K}(s,t)-Q^{z}_{K'}(s,t) = 2\,sgn(c)(s^{z(c)}t^{p^{0}(c)}-1)(s^{z(c)}t^{p^{1}(c)}-1).\]
\end{lemma}
\begin{proof}
	Let $K'$ be obtained from $K$ by a crossing change operation at a crossing $c \in C(K)$, and let $c'$ be the corresponding crossing in $K'$, with $sgn(c')=-sgn(c)$ (Fig.~\ref{fig:ccmove}).
  \begin{figure}[htbp]
	\centering
	\subfloat[\label{fig:ccmove}]{\includegraphics[width=0.26\textwidth]{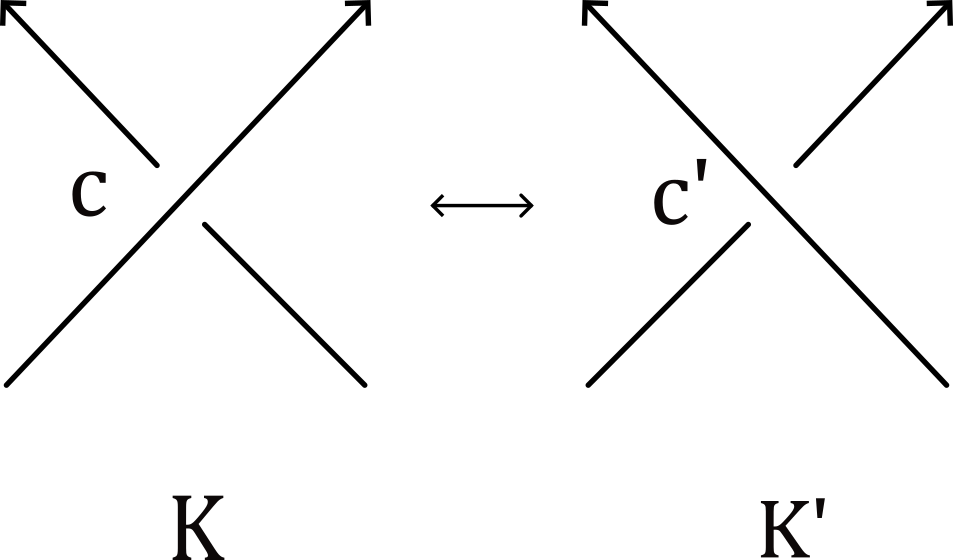}}
	\qquad
	\subfloat[\label{fig:ccsmc}]{\includegraphics[width=0.26
		\textwidth]{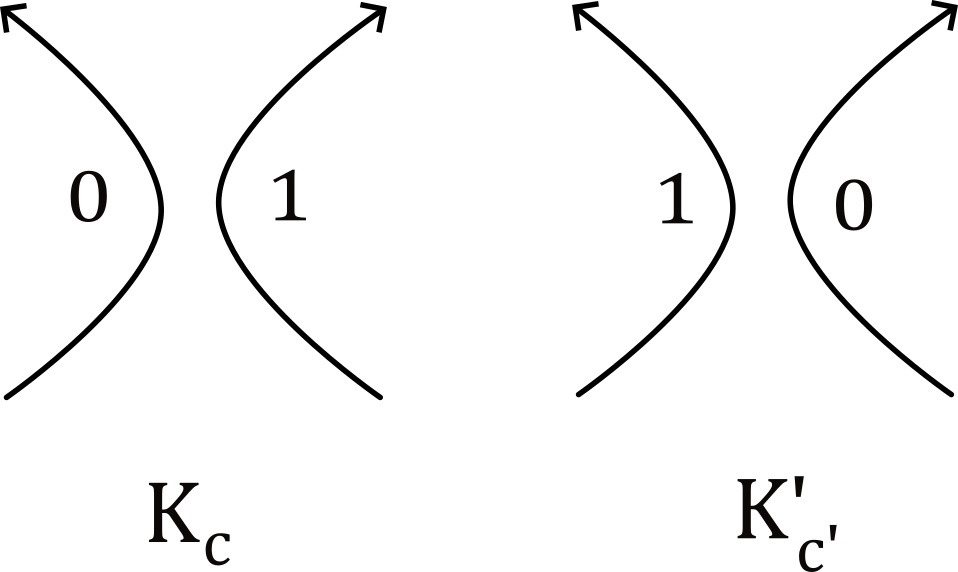}}
	\qquad
	\subfloat[\label{fig:ccsme}]{\includegraphics[width=0.28
		\textwidth]{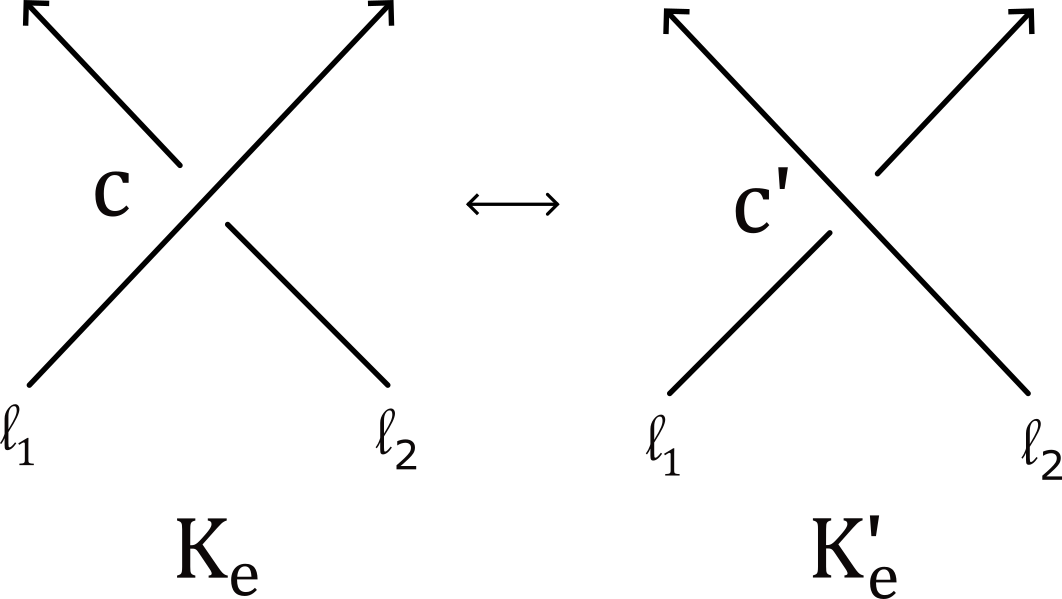}}
	\caption{Crossing change move at a crossing $c\in C(K)$ and the smoothing components.}
\end{figure}
	We can see from Fig.~\ref{fig:ccsmc} that the diagrams $K_{c}$ and $K'_{c'}$ have opposite labels on the components. Therefore, $$ind^{\,0}(c,K)= ind^{\,1}(c',K')\quad \text{and}\, ind^{\,1}(c,K)= ind^{\,0}(c',K').$$ Since  $K$ and $K'$ are twisted  knot diagrams, we have $$ind^{\,1}(c,K)=-ind^{\,0}(c,K)= -ind^{\,1}(c',K')= ind^{\,0}(c',K').$$ 
	Therefore,
	\[
	c\in ZC(K)
	\iff
	c^{\prime}\in ZC(K').
	\]
	
	Now, for all crossings $e \in C(K) \setminus \{c\}= C(K')\setminus \{c'\}$,  $K_{e}$ and $K'_{e}$ differ locally as shown in Fig.~\ref{fig:ccsme}. If $l_{1}\neq l_{2}$ then for $l \in \{0,1\}$,
	\[	|ind ^{\,l}(e,K)-ind ^{\,l}(e,K')|=|sgn(c)+sgn(c')|=0,\]
	and if $l_{1}= l_{2}$ then $K_{e}$ and $K'_{e}$ are identical. In both cases, we have  $e \in ZC(K)$ if and only if $e' \in ZC(K')$.
	Then we see that $z(e)$ do not change under $T3a$ move when $ e\in ZC(K)$.
Also $z(c) = z(c')$ when $c \in ZC(K)$.
Moreover, $p^{l}(e)$, $l \in \{0,1\}$ do not change for all $ e \in C(K)\setminus \{c\}= C(K')\setminus \{c'\}$. And
	\[p^{0}(c)=p^{1}(c'), \quad p^{1}(c)=p^{0}(c').\]
	Hence, when $c \in ZC(K)$ we have,
	\begin{align*}
		Q^{z}_{K}(s,t)-Q^{z}_{K'}(s,t) =&sgn(c)(s^{z(c)}t^{p^{0}(c)}-1)(s^{z(c)}t^{p^{1}(c)}-1)\\ &-sgn(c')(s^{z(c')}t^{p^{0}(c')}-1)(s^{z(c')}t^{p^{1}(c')}-1),\\
		=&sgn(c)(s^{z(c)}t^{p^{0}(c)}-1)(s^{z(c)}t^{p^{1}(c)}-1)\\ &+sgn(c)(s^{z(c)}t^{p^{1}(c)}-1)(s^{z(c)}t^{p^{0}(c)}-1),\\
		=& 2\,sgn(c)(s^{z(c)}t^{p^{0}(c)}-1)(s^{z(c)}t^{p^{1}(c)}-1).
	\end{align*}
and	when $c \notin ZC(K)$ we have,
	\begin{align*}
		Q^{z}_{K}(s,t)-Q^{z}_{K'}(s,t) =&sgn(c)(s\,\,t^{p^{0}(c)}-1)(s\,\,t^{p^{1}(c)}-1)\\ &-sgn(c')(s\,t^{p^{0}(c')}-1)(s\,t^{p^{1}(c')}-1),\\
		=&sgn(c)(s\,\,t^{p^{0}(c)}-1)(s\,\,t^{p^{1}(c)}-1)\\ &+sgn(c)(s\,\,t^{p^{1}(c)}-1)(s\,\,t^{p^{0}(c)}-1),\\
		=& 2\,sgn(c)(s\,\,t^{p^{0}(c)}-1)(s\,\,t^{p^{1}(c)}-1).
	\end{align*}
\end{proof}

\begin{theorem}\label{thm:homotopic}
	Let $K$ and $K'$ be two homotopic twisted knots. Then there exists $m
	\in \mathbb{Z}$ such that
	\begin{equation*}
	\begin{aligned}
		Q^{z}_{K}(s,t)-Q^{z}_{K'}(s,t) = &\sum_{i=1}^{m} 2\,a_{i}(s^{b_{i}}t^{p^{0}_{i}}-1)(s^{b_{i}}t^{p^{1}_{i}}-1),
	\end{aligned}
\end{equation*}
where  $a_{i}\in \{-1,1\}$, and $b_{i}, p^{0}_{i}, p^{1}_{i} \in \{0,1\}$, for $i = 1,\cdots, m$.
\end{theorem}
  \begin{proof}
  Since $K$ and $K'$ are homotopic, there exist diagrams $D$ and $D'$ of $K$ and $K'$ respectively such that $D'$ is obtained from $D$ by $k$ crossing-change operations.
  	Suppose these crossing changes are performed at the classical crossings
  	$c_1,c_2,\ldots,c_k$ of $K$.

 Without loss of generality, assume that the first $n(\leq k)$ number of crossings are zero crossings. Therefore, by Lemma~\ref{lem:ccmove} we can write
  	\begin{equation}\label{eq:dist}
  		\begin{aligned}
  			Q^{z}_{K}(s,t)-Q^{z}_{K'}(s,t) =& Q^{z}_{D}(s,t)-Q^{z}_{D'}(s,t),\\
  			& =\sum_{i=1}^{n} 2\,sgn(c_{i})(s^{z(c_{i})}t^{p^{0}(c_{i})}-1)(s^{z(c_{i})}t^{p^{1}(c_{i})}-1) \\
  			&+ \sum_{i=n+1}^{k} 2\,sgn(c_{i})(s\,t^{p^{0}(c_{i})}-1)(s\,t^{p^{1}(c_{i})}-1).
  		\end{aligned}
  	\end{equation}
  
  Some of these polynomial terms may cancel, and some terms in the first summand may be zero, so there exists an integer $m(\leq k)$ such that we can write
  	\begin{equation*}
  		\begin{aligned}
  			Q^{z}_{K}(s,t)-Q^{z}_{K'}(s,t) = &\sum_{i=1}^{m} 2\,a_{i}(s^{b_{i}}t^{p^{0}_{i}}-1)(s^{b_{i}}t^{p^{1}_{i}}-1),
  		\end{aligned}
  	\end{equation*}
  		where  $a_{i}\in \{-1,1\}$, and $b_{i}, p^{0}_{i}, p^{1}_{i} \in \{0,1\}$, for $i = 1, \cdots, m (\leq k)$. 
  		 \end{proof}
  		 \begin{remark}
  		 Since  each crossing change contributes exactly one term in Equation~\ref{eq:dist}, there are no more than $k$ terms with $|a_{i}|=1$. Therefore, $
  		 \sum_{i =1}^{m}|a_i|
  		 \le k.$
  		  Therefore, when we consider the minimal sequence of crossing change operations then we get
  		   \[\sum_{i=1}^{m}|a_i|
  		 \le d_{G}(K,K').\]
  		 This relation provides a lower bound on the Gordian distance of two homotopic twisted knots.
  		
  		 \end{remark}
 \begin{example}
	Let $K$ and $K'$ be two homotopic twisted knots. From Fig.~\ref{fig:cc_exm}, we get $d_{G}(K,K') \leq 2.$
	Now, 
	\begin{align*}
	Q^{z}_{K}(s,t)=& -2(st-1)^{2}-2(s-1)^2+2(t-1)^2,\\
\text{and}\quad	Q^{z}_{K'}(s,t)=&\phantom{-}2(t-1)^2.
	\end{align*}
	Therefore, \[Q^{z}_{K}(s,t)-Q^{z}_{K'}(s,t)=-2(st-1)^{2}-2(s-1)^{2}. \] By Theorem~\ref{thm:homotopic}, $d_{G}(K,K') \geq 2$. Hence $d_{G}(K,K'
	) = 2$.
\end{example}	
	
\begin{figure}[htbp]
	\centering 
	\includegraphics[width=\textwidth]{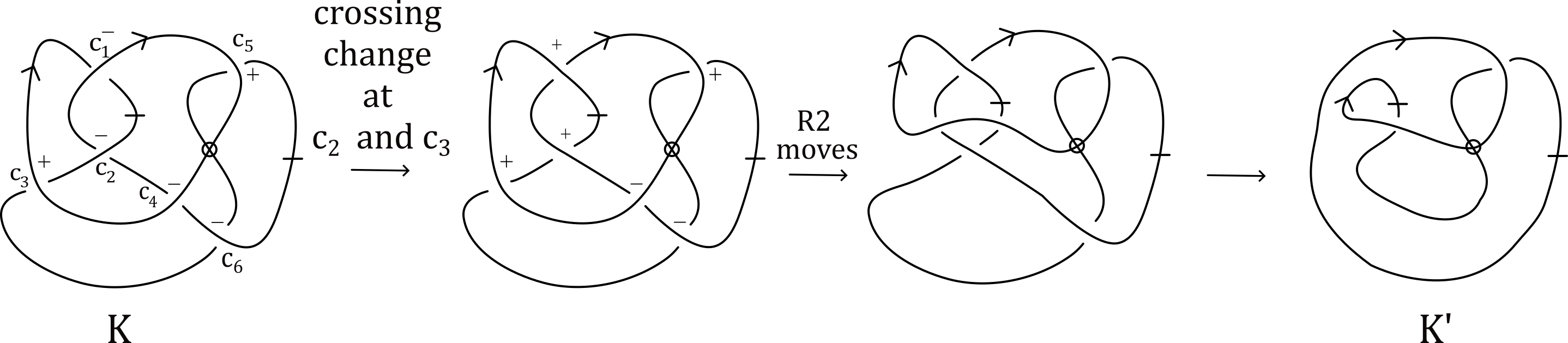}
	\caption{}
	\label{fig:cc_exm}
\end{figure}

 \begin{remark}
When a twisted knot $K$ is homotopic to one of the trivial twisted knot $K_{0}$, then the Gordian distance $d_{G}(K, K_{0})$ turns out to be the  unknotting number for the twisted knot $K$, denoted by $u(K)$. See Example~\ref{rem:unknot}.
 \end{remark}
 
  \begin{example}\label{rem:unknot}
  	Let $K$ be the twisted knot diagram in the Fig.~\ref{fig:gor1} which is homotopic to the trivial twisted knot $K_0$ without a bar. Now, $Q^{z}_{K}(s,t)= -2(st-1)^{2}$. Compare this expression with $Q^{z}_{K_{0}}(s,t) = 0$. By Theorem~\ref{thm:homotopic}, $d_{G}(K,K_{0}) \geq 1$. Also, from the Fig.~\ref{fig:gor2}, we see that $d_{G}(K,K_{0}) \leq 1$. Hence $u(K)= d_{G}(K,K_{0}) = 1$.
  		\begin{figure}[htbp]
  		\centering 
  		\includegraphics[width=0.18\textwidth]{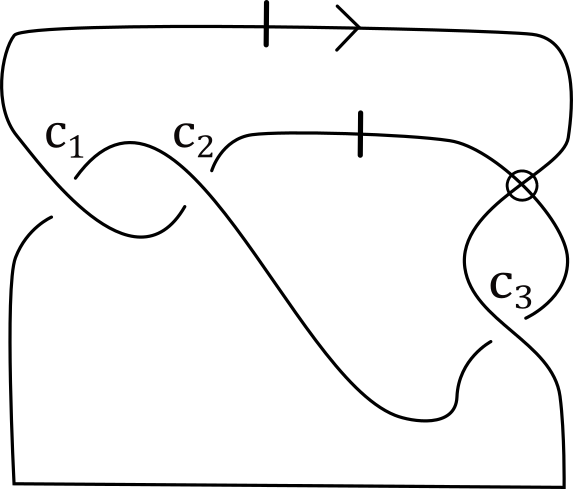}
  		\caption{A non-trivial twisted knot diagram $K$.}
  		\label{fig:gor1}
  	\end{figure}
  		\begin{figure}[htbp]
  		\centering 
  		\includegraphics[width=0.9\textwidth]{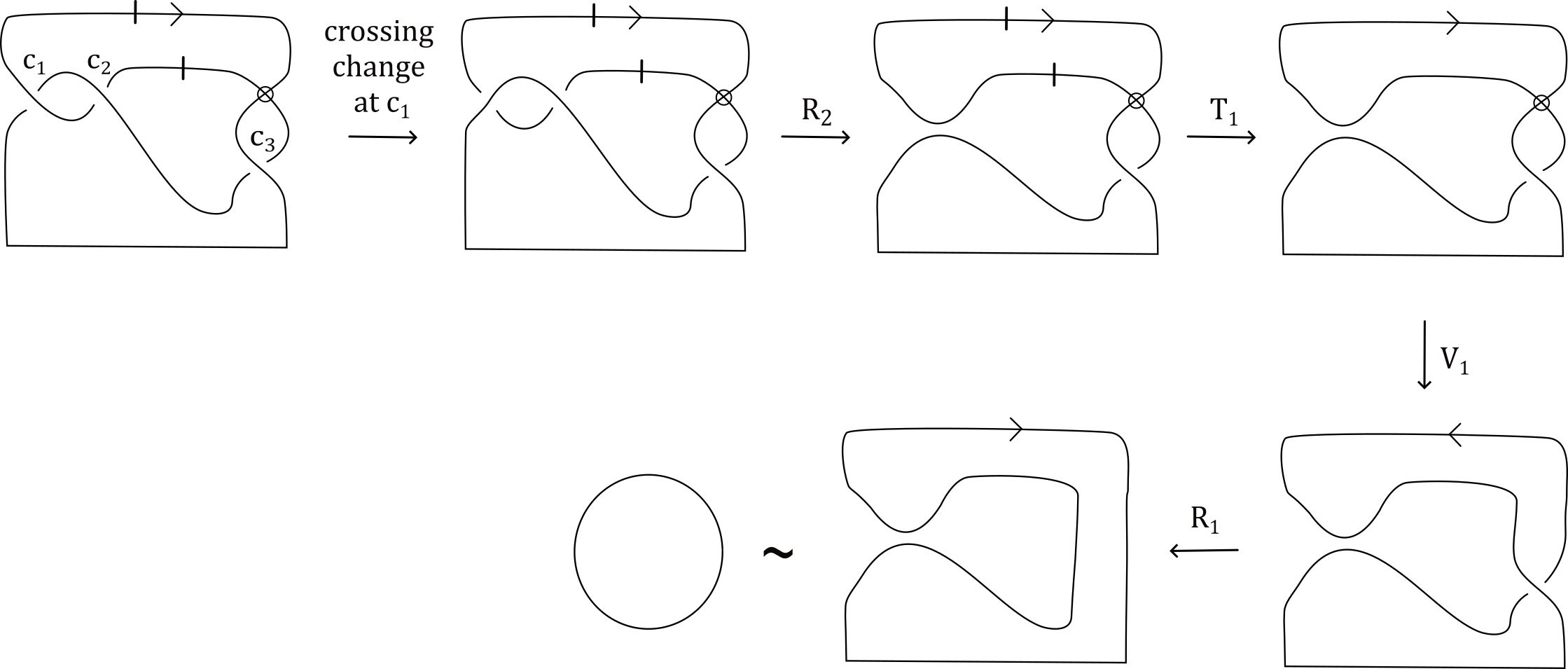}
  		\caption{A sequence of one crossing change and generalized Reidemeister moves transforming $K$ into the trivial twisted knot $K_0$, showing $d_G(K,K_0)\le 1$.}
  		\label{fig:gor2}
  	\end{figure}
  \end{example}
  \begin{remark}
  	There are two distinct homotopy classes for two trivial twisted knots. Moreover, there are more than two homotopy classes for twisted knots. 
  	 For $K$ and $K'$,
  	\[Q_{K}^{z}(s,t)=(t-1)^{2}, \quad Q_{K'}^{z}(s,t)=(s-1)^{2}+(st-1)^{2}.   \]
  	Therefore, by computing the $Q^z$-polynomial we can see that $K$ and $K'$ in Fig.~\ref{fig:hmcls} are not homotopic to a trivial twisted knot diagram with or without bar. And by Theorem~\ref{thm:homotopic}, $K$ and $K'$ are not homotopic. 
  	\begin{figure}[htbp]
  		\centering 
  		\includegraphics[width=0.4\textwidth]{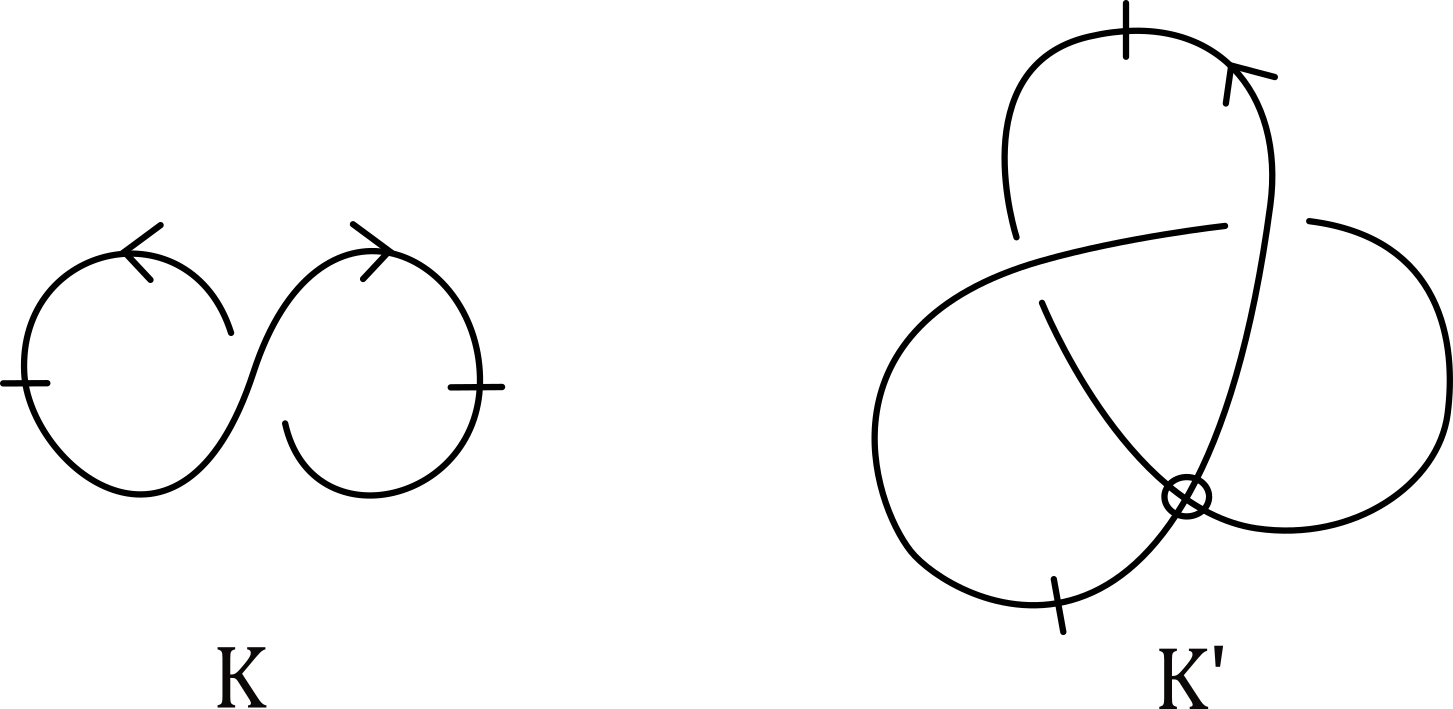}
  		\caption{}
  		\label{fig:hmcls}
  	\end{figure}
  \end{remark}
  \subsection{Cosmetic crossing  change conjecture}
A crossing in a knot diagram is nugatory if there is a circle in the plane of the diagram that intersects the diagram transversely at that single crossing and does not have any other intersection points with the diagram. In Fig.~\ref{fig:nug}, the dotted circles show the nugatory crossings in the knot diagram. Changing this crossing does not change the knot type. A non-nugatory crossing whose change does not change the knot type is called a cosmetic crossing.
	\begin{figure}[htbp]
	\centering 
	\includegraphics[width=0.3\textwidth]{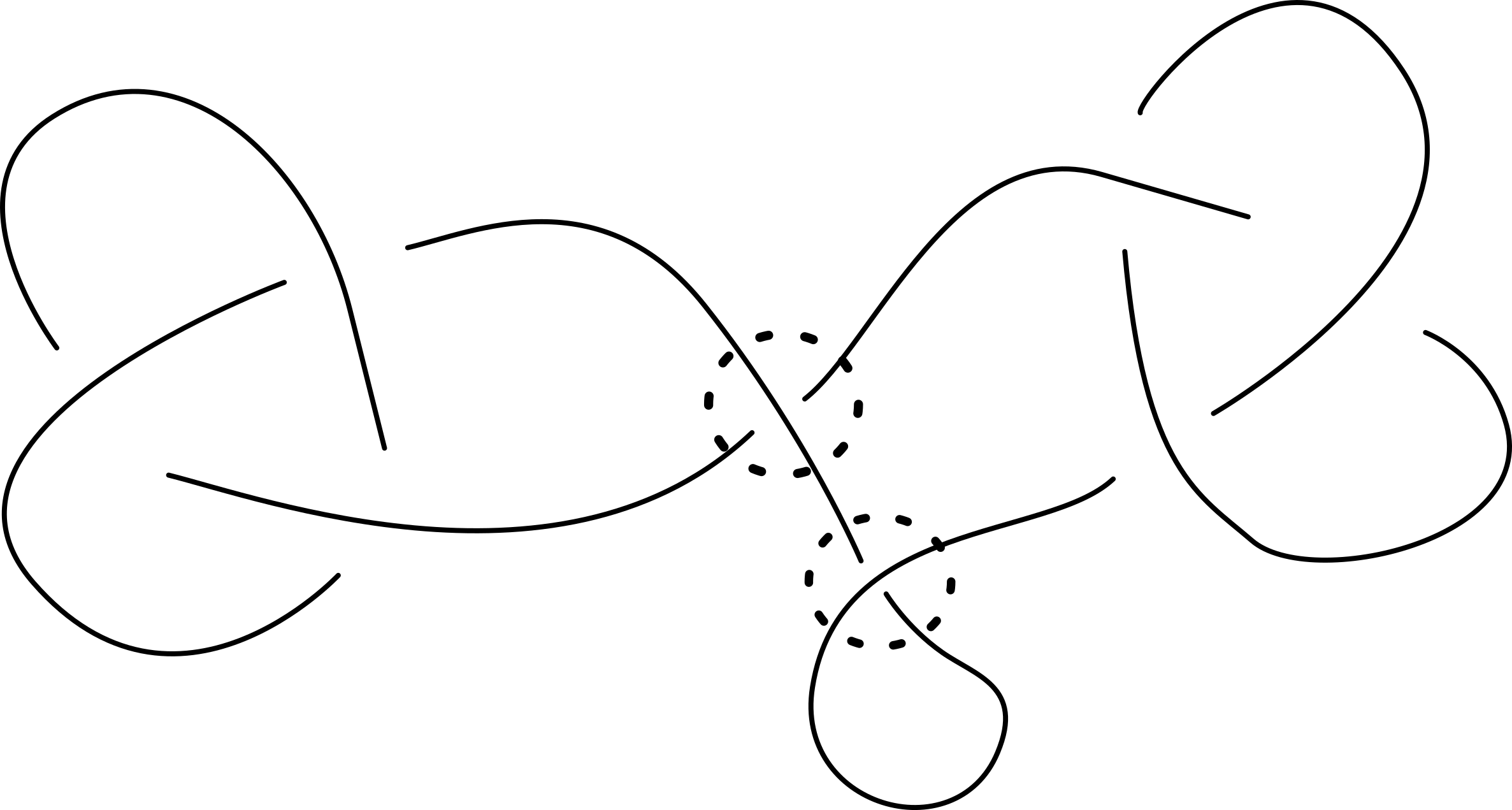}
	\caption{Examples of nugatory crossing.}
	\label{fig:nug}
\end{figure}
The following conjecture is attributed to \cite{Kirby} and has been proven true for several classes of  classical knots \cite{Ichi, Balm, Cre} and virtual knots \cite{Fol}. But the situation for twisted knots remains largely unexplored. We will study the existence of cosmetic crossings in a twisted knot diagram.

\begin{conjecture}\cite{Kirby}(Cosmetic crossing conjecture)
If $K$ admits a crossing change at a crossing $c$ which preserves the oriented isotopy class of the knot, then $c$ 
is nugatory.
\end{conjecture}

\begin{theorem}\label{thm:oddcosmetic}
Let $K$ be a twisted knot diagram with every classical crossing $
c$ of $K$ satisfying one of the following conditions.
	\begin{enumerate}
		
		\item[\rm (i)] $c \notin ZC(K)$ i.e., $ind^{\,0}(c,K)$ is odd 
		
		\item[\rm (ii)] $c \in ZC(K)$
		and $zind^{\,0}(c,K)$ is odd.
		\end{enumerate}
	Then, $K$ does not admit a cosmetic crossing change.
\end{theorem}
 \begin{proof}
  	Let a twisted knot diagram $K'$ is obtained from $K$ by a crossing change at a crossing $c$. If $c\notin ZC(K)$, or  $c \in ZC(K)$ with $z(c)=1$, then in both cases we can write the following using Lemma~\ref{lem:ccmove},
  		\begin{align*}
  		Q^{z}_{K}(s,t)-Q^{z}_{K'}(s,t)
  		&=2\,sgn(c)(s\,t^{p^{0}(c)}-1)(s\,t^{p^{1}(c)}-1)\neq 0,
  	\end{align*}
  	where $p^{0}(c),\, p^{0}(c)\, \in \{0,1\}$.
  	And so $K'$ is not isotopic to $K$.
  	Therefore, crossing change at $c$ does not change the isotopy of $K$. Hence the proof.
  \end{proof}
  \begin{figure}[htbp]
  	\centering
  	\includegraphics[width=0.5\textwidth]{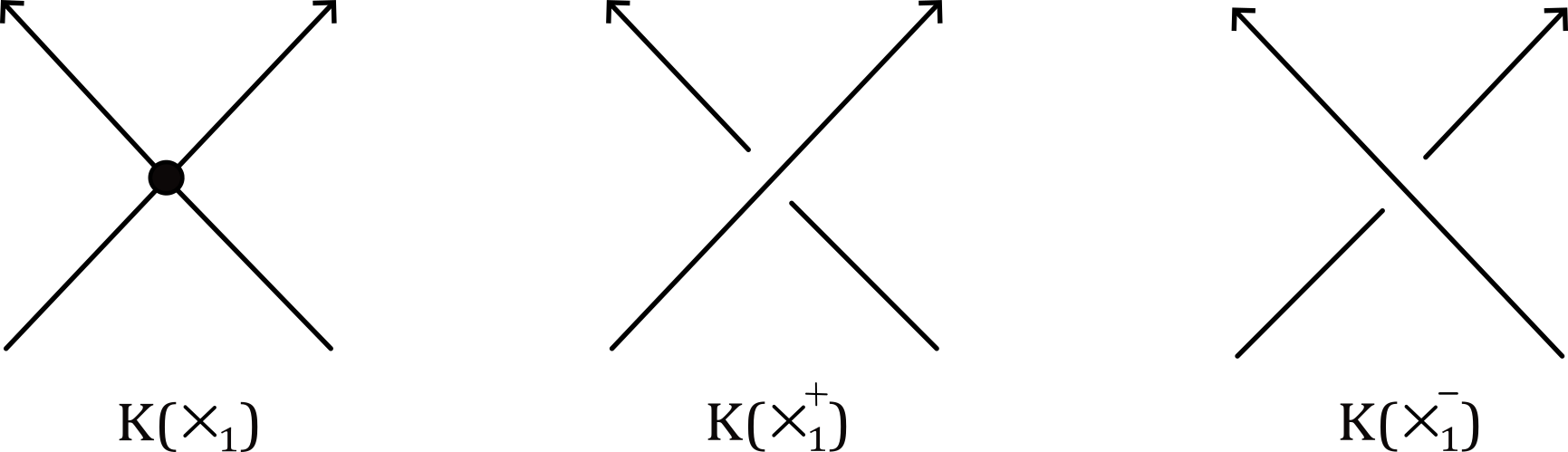}
  	\caption{Resolution of a singular crossing into positive and negative crossings, used in the Vassiliev skein relation.}
  	\label{fig:singres}
  \end{figure}
\subsection{The \texorpdfstring{$Q^{z}$}{Qz}-polynomial is a Vassiliev invariant of order one}
	Let $v$ be an invariant of a twisted knot, and let $v$ take values in an abelian group $G$. Let $K(\times_{1})$ is a singular twisted knot diagram with one singular crossing and $K(\times_{1}^{+})$, $K(\times_{1}^{-})$ correspond to the twisted knot diagram where the singular crossing is replaced by a positive crossing and a negative crossing, respectively (Fig.~\ref{fig:singres}). A twisted knot invariant $v$ can be extended to a singular twisted knot invariant by using the following Vassiliev skein relation \cite{Vass}. $$v(K(\times_{1}))= v(K(\times_{1}^{+}))-v(K(\times_{1}^{-})).$$

\begin{theorem}
	$Q^{z}_{K}(s,t)$ is a Vassiliev invariant of order one.
\end{theorem}
\begin{proof}
	Let $K(\times_{1}, \times_{2})$ be a twisted knot diagram with two singular crossings $\times_{1}, \times_{2}$. To show that $Q^{z}_{K}(s,t)$ is a Vassiliev invariant of order one, we first prove that $Q^{z}_{K(\times_{1}, \times_{2})}(s,t)=0$.
	Using the Vassiliev skein relation, we can write
	\begin{align*}
		Q^{z}_{K(\times_{1}, \times_{2})}(s,t)=& Q^{z}_{K(\times_{1}, \times_{2}^{+})}(s,t)-Q^{z}_{K(\times_{1}, \times_{2}^{-})}(s,t),\\
		=& Q^{z}_{K(\times_{1}^{+}, \times_{2}^{+})}(s,t)-Q^{z}_{K(\times_{1}^{+}, \times_{2}^{-})}(s,t)\\
		&-Q^{z}_{K(\times_{1}^{-}, \times_{2}^{+})}(s,t)+Q^{z}_{K(\times_{1}^{-}, \times_{2}^{-})}(s,t),
	\end{align*}		
where $\times_{i}^{+},\,i=1,2$ (resp. $\times_{i}^{-}$) denotes that the singular crossing $\times_{i},\,i=1,2$ is replaced by a positive (resp. negative) classical crossing.
Notice that \[|ind^{\,0} ( \times_{2}^{+}, K(\times_{1}^{+}, \times_{2}^{+}) ) - ind^{\,0}(\times_{2}^{+}, K(\times_{1}^{-}, \times_{2}^{+}) )| =0\, \text{or}\,2,\]
depending on whether
 $\times_{1}^{+}$ and $\times_{1}^{-}$ are self crossings of the smoothing components of $\times_{2}^{+}$ in $ K(\times_{1}^{+}, \times_{2}^{+})$ and $ K(\times_{1}^{-}, \times_{2}^{+})$ or not. 

 Therefore, $$ \times_{2}^{+} \in ZC(K(\times_{1}^{+}, \times_{2}^{+}))\;\; \text{ if and only if } \;\;\times_{2}^{+} \in ZC(K(\times_{1}^{-}, \times_{2}^{+})).$$
 
  Moreover, the bar parity in the components are same in the smoothing components of $K(\times_{1}^{+}, \times_{2}^{+})  $ and $K(\times_{1}^{-}, \times_{2}^{+}) $ at $\times_{2}^{+}$.

Case 1: If the crossing $\times_{2}^{+} \in ZC(K(\times_{1}^{+}, \times_{2}^{+}))$ then by Theorem~\ref{lem:ccmove},
	\[Q^{z}_{K(\times_{1}^{+}, \times_{2}^{+})}(s,t)-Q^{z}_{K(\times_{1}^{+}, \times_{2}^{-})}(s,t)= 2(s^{z(\times_{2}^{+})}t^{p^{0}(\times_{2}^{+})}-1)(s^{z(\times_{2}^{+})}t^{p^{1}(\times_{2}^{+})}-1),\]
and
	\[Q^{z}_{K(\times_{1}^{-}, \times_{2}^{+})}(s,t)-Q^{z}_{K(\times_{1}^{-}, \times_{2}^{-})}(s,t)= 2(s^{z(\times_{2}^{+})}t^{p^{0}( \times_{2}^{+})}-1)(s^{z(\times_{2}^{+})}t^{p^{1}(\times_{2}^{+})}-1).\]

Case 2: If the crossing $\times_{2}^{+} \notin ZC(K(\times_{1}^{+}, \times_{2}^{+}))$ then by Theorem~\ref{lem:ccmove},
\[Q^{z}_{K(\times_{1}^{+}, \times_{2}^{+})}(s,t)-Q^{z}_{K(\times_{1}^{+}, \times_{2}^{-})}(s,t)= 2(s\,t^{p^{0}(\times_{2}^{+})}-1)(s\,t^{p^{1}(\times_{2}^{+})}-1),\]
and
\[Q^{z}_{K(\times_{1}^{-}, \times_{2}^{+})}(s,t)-Q^{z}_{K(\times_{1}^{-}, \times_{2}^{-})}(s,t)= -2(s\,t^{p^{0}(\times_{2}^{+})}-1)(s\,t^{p^{1}(\times_{2}^{+})}-1).\]
Therefore, in both cases, \[	Q^{z}_{K(\times_{1}, \times_{2})}(s,t)=0.\]
Hence, $Q^{z}_{K}(s,t)$ is a Vassiliev invariant of order $\leq 1$. 
\begin{figure}[htbp]
	\centering
	\includegraphics[width=0.8\textwidth]{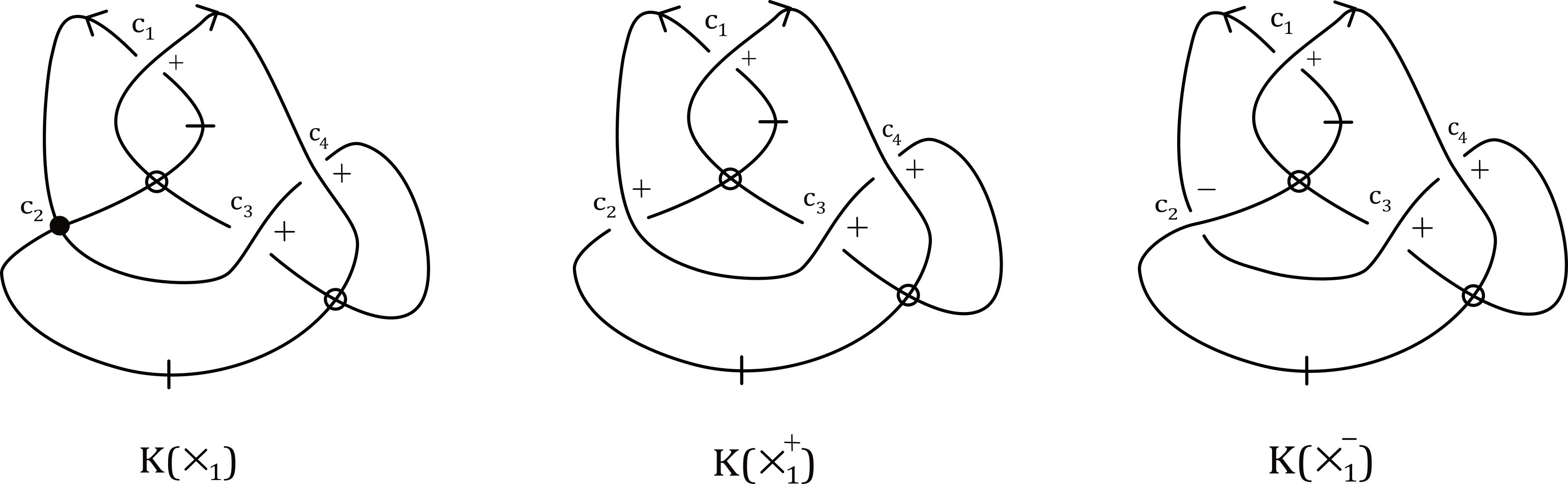}
	\caption{The singular twisted knot diagram $K(\times_{1})$ and its resolutions $K(\times_{1}^{+})$ and $K(\times_{1}^{-})$.}
	\label{fig:vasex}
\end{figure}

Now we consider the singular twisted knot diagram $K(\times_{1})$ in Fig.~\ref{fig:vasex}. Then 
\[Q^{z}_{K(\times_{1})}(s,t)= Q^{z}_{K(\times_{1}^{+})}(s,t) - Q^{z}_{K(\times_{1}^{-})}(s,t).\]
where $K(\times_{1}^{+})$ and $K(\times_{1}^{-})$ denote the diagrams resulting from resolving the singular crossing $c$ of $K(\times_{1})$ into a positive and a negative crossing, respectively, see Fig.~\ref{fig:vasex}.

Now,  by example~\ref{ex:vasdia}, we know that $Q^{z}_{K(\times_{1}^{+})}(s,t)= 3(s-1)^{2}+ (st-1)^{2}$. The values of $z(c), p^{0}(c)$, and $p^{1}(c)$ for the diagram $K(\times_{1}^{-})$ are given in Table~\ref{tab:vasex}. Therefore,
\[Q^{z}_{K(\times_{1}^{-})}(s,t)=3(s-1)^{2}- (st-1)^{2}.\]
Calculations of $Q^{z}_{K(\times_{1}^{+})}(s,t)$ and $Q^{z}_{K(\times_{1}^{-})}(s,t)$ imply that \[Q^{z}_{K(\times_{1})}(s,t) =Q^{z}_{K(\times_{1}^{+})}(s,t)-Q^{z}_{K(\times_{1}^{-})}(s,t) = 2(st-1)^{2} \neq 0.\] Thus, $Q^{z}_{K}(s,t)$ is a Vassiliev invariant of order one.

\begin{table}[htbp]
	\centering
	\resizebox{0.5\textwidth}{!}{
	\begin{tabular}{|c|c|c|c|c|c|} 
		\hline
		$c$ & $sgn(c)$ & $ind^{\,0}(c)$  & $z(c)$ & $p^{0}(c)$ & $p^{1}(c)$ \\ 
		\hline 
		$c'_{1}$ & $+1$ & $0$  & $1$ &  $0$ & $0$ \\ 
		\hline
		$c'_{2}$ & $-1$ &  $-1$ &  $\star$ & $1$ & $1$ \\
		\hline
		$c'_{3}$ &$+1$ &  $-1$  & $\star$ & $0$ & $0$ \\
		\hline   
		$c'_{4}$ &$+1$ &  $0$  &$1$ & $0$ & $0$ \\
		\hline 
	\end{tabular}}
	\caption{Values of $z(c), p^0(c), p^1(c)$ for the crossings of $K(\times_{1}^{-})$.}
	\label{tab:vasex}
\end{table}
 \end{proof}
\section*{Statements and Declarations}

\textbf{Author contributions} T. Mahato and P. Madeti contributed to
conceptualization, methodology, investigation, and writing. T. Mahato performed the formal analysis and prepared the
manuscript. P. Madeti contributed to supervision, validation, and
review \& editing.
\vspace{3mm}

\textbf{Funding} This work was supported by the Anusandhan National
Research Foundation (ANRF), Government of India (Grant no.
CRG/2023/004921/343).
\vspace{3mm}

\textbf{Data Availability Statement} No datasets were generated or
analysed during the current study.
\vspace{3mm}

\textbf{Declarations}
\textbf{Conflict of interest} The authors declare no Conflict of
interest.
	\bibliographystyle{plain}

\end{document}